\documentclass[preprint,imslayout]{imsart}
\pdfoutput=1

\usepackage{amsmath, amsfonts, amsthm} 
\usepackage{mathtools}
\usepackage[numbers]{natbib}
\usepackage{graphicx}
\usepackage{subcaption}
\usepackage{hyperref}
\usepackage{enumitem}
\usepackage[makeroom]{cancel}
\usepackage{epstopdf}

\startlocaldefs
\RequirePackage{amsmath}
\RequirePackage{xparse}
\RequirePackage{dsfont}

\newcommand{\bull}{\textbullet\ }

\newcommand{\dd}{\mathrm{d}}

\newcommand{\rec}[1]{\frac{1}{#1}}

\newcommand{\mini}{\wedge}

\newcommand{\RR}{{\mathds{R}}}     

\newcommand{\I}{{\mathds{1}}}      

\newcommand{\II}[1]{{\mathds{1}\!\left\{#1\right\}}}

\newcommand{\defEqual}{=}

\newcommand{\distEqual}{\overset{\text{d}}{=}}
\newcommand{\iid}{i.i.d.\ }

\makeatletter
\DeclareDocumentCommand\E{m o o o o}{
	\mathds{E}%
	\def\@tempa{#4} \ifx\@tempa\@empty\relax{}\else \IfNoValueF{#4}{_{#4}}\fi
	\def\@tempa{#5} \ifx\@tempa\@empty\relax{}\else \IfNoValueF{#5}{^{#5}}\fi
	\!\left[
		#1
		\def\@tempa{#2} \ifx\@tempa\@empty\relax{}\else \IfNoValueF{#2}{\, | \, #2} \fi
		\def\@tempa{#3} \ifx\@tempa\@empty\relax{}\else \IfNoValueF{#3}{\, ; \, #3} \fi
	\right]
}

\DeclareDocumentCommand\PP{m o o o o o}{
	\mathds{P}%
	\def\@tempa{#2} \ifx\@tempa\@empty\relax{}\else \IfNoValueF{#2}{_{#2}}\fi
	\def\@tempa{#3} \ifx\@tempa\@empty\relax{}\else \IfNoValueF{#3}{^{#3}}\fi
	\!\left[
		#1
		\def\@tempa{#4} \ifx\@tempa\@empty\relax{}\else \IfNoValueF{#4}{\, | \, #4} \fi
		\def\@tempa{#5} \ifx\@tempa\@empty\relax{}\else \IfNoValueF{#5}{\, ; \, #5} \fi
	\right]
}
\makeatother

\newcommand{\dens}{g}       
\newcommand{\dist}{G}       
\newcommand{\rdist}{\overline{\dist}} 

\newcommand{\GammaD}[2]{{\mathrm{Gamma}\!\left( #1, #2 \right)}}

\newcommand{\M}{\mathcal{M}}
\newcommand{\PPSpace}[1]{ \M^{\#}(#1)}

\newcommand{\AveSojourn}{\theta}

\makeatletter
\DeclareDocumentCommand\LF{o o}{
	\mathcal{L}%
	\def\@tempa{#2} \ifx\@tempa\@empty\relax \else \IfNoValueF{#2}{_{#2}}\fi
	\def\@tempa{#1} \ifx\@tempa\@empty\relax \else \IfNoValueF{#1}{\!\left[#1 \right]}\fi
}

\DeclareDocumentCommand\expLF{o o}{
	\mathcal{\ell}%
	\def\@tempa{#2} \ifx\@tempa\@empty\relax \else \IfNoValueF{#2}{_{#2}}\fi
	\def\@tempa{#1} \ifx\@tempa\@empty\relax \else \IfNoValueF{#1}{\!\left[#1 \right]}\fi
}
\makeatother

\DeclareDocumentCommand\Palm{m}{Q_{#1}}
\DeclareDocumentCommand\APalm{m}{\overline{Q}_{#1}}

\newcommand{\CSz}{C}

\newcommand{\NMiss}{\mu}       

\newcommand{\Error}{e}

\newcommand{\AveNMiss}{m}

\newcommand{\IAveNMiss}{M}

\newcommand{\MinAveNMiss}{m_0} 

\newcommand{\Req}{\Theta}   

\newcommand{\NReqs}{N}      

\newcommand{\Pop}{R}        

\newcommand{\HitProba}[1]{q_{#1}}
\newcommand{\MissProba}[1]{p_{#1}}

\newcommand{\Lifespan}{L} 

\newcommand{\CatPP}{\Gamma^g}

\newcommand{\CatArrRate}{\gamma}

\newcommand{\MarkCatPP}{\widetilde{\Gamma}}

\newcommand{\CatArr}{a}

\newcommand{\DocReqPP}{\xi}

\newcommand{\DocReqInt}{\lambda}

\newcommand{\DocReqMeas}{\Lambda}

\newcommand{\DocReqMeanFn}{\Lambda}

\newcommand{\DocReqMeanComp}{\bar{\Lambda}}

\newcommand{\AveNReqs}{\Lambda}  

\newcommand{\NDifDocs}{X}

\newcommand{\NDifDocsMean}{\Xi}

\newcommand{\TotReqPP}{\Gamma}

\newcommand{\iReq}{r}

\newcommand{\ExitTime}[1]{T_{#1}}

\newcommand{\ChT}[1]{t_{#1}}

\newcommand{\FPTPoiss}[1]{\widehat{T}_{#1}}

\newtheoremstyle{break}
  {\topsep}{\topsep}%
  {\itshape}{}%
  {\bfseries}{}%
  {\newline}{}%

\theoremstyle{break}
\newtheorem{thm}{Theorem}

\newtheorem{rmk}[thm]{Remark}
\newtheorem{lem}[thm]{Lemma}
\newtheorem{pro}[thm]{Proposition}

\endlocaldefs

\begin{document}

\begin{frontmatter}
  \title{Cache Miss Estimation for Non-Stationary Request Processes}
  \runtitle{Cache Miss Estimation}

  \begin{aug}
    \author{
      \fnms{Felipe} \snm{Olmos}\thanksref{t1,t2},
      \ead[label=e1]{luisfelipe.olmosmarchant@orange.com}
    }
    \author{
      \fnms{Carl} \snm{Graham}\thanksref{t2}
      \ead[label=e2]{carl.graham@polytechnique.edu}
    }
    \and
    \author{
      \fnms{Alain} \snm{Simonian}\thanksref{t1}
      \ead[label=e3]{alain.simonian@orange.com}
    }
    \runauthor{F.Olmos, C.Graham, A.Simonian}
    \address{
      Orange Labs\\
      Department OLN / NMP / TRM \\
      38 - 40 rue du G\'en\'eral Leclerc\\
      92794 Issy-Les-Moulineaux, France\\
      \printead{e1}
      \\
      \printead{e3}
    }
    \address{
      Centre de Math\'ematiques Appliqu\'ees\\
      \'Ecole Polytechnique, CNRS\\
      Universit\'e Paris Saclay\\
      Route de Saclay\\
      91128 Palaiseau, France\\
      \printead{e2}
    }
  \end{aug}

  \affiliation{Orange Labs\thanksmark{t1} and CMAP, \'Ecole Polytechnique\thanksmark{t2}}

  \begin{abstract}
    The goal of the paper is to evaluate the miss probability of a Least Recently
    Used (LRU) cache, when it is offered a non-stationary request process given
    by a Poisson cluster point process. First, we construct a probability space
    using Palm theory, describing how to consider a tagged document with respect
    to the rest of the request process. This framework allows us to derive a fundamental
    integral formula for the expected number of misses of the tagged
    document. Then, we consider the limit when the cache size and the arrival
    rate go to infinity in proportion, and use the integral formula to derive an
    asymptotic expansion of the miss probability in powers of the inverse of the
    cache size. This enables us to quantify and improve the accuracy of the
    so-called \emph{Che approximation}.
  \end{abstract}

  \begin{keyword}[class=MSC]
    \kwd{68B20}
    \kwd{60G55}
    \kwd{60K30}
  \end{keyword}

  \begin{keyword}
    \kwd{performance evaluation}
    \kwd{LRU cache policy}
    \kwd{Poisson cluster process}
    \kwd{Cox process}
    \kwd{scaling limit expansion}
    \kwd{Che approximation}
  \end{keyword}

\end{frontmatter}
\section{Introduction}
Since the early days of the Web, cache servers have been used to provide users
faster document retrieval while saving network resources. In recent years,
there has been a renewed interest in the study of these systems, since they are
the building bricks of \emph{Content Delivery Networks} (CDNs), a key
component of today's Internet. In fact, these systems handle nowadays around
$60\%$ of all video traffic, and it is predicted that this quantity will
increase to more than $70\%$ by 2019~\cite{cisco2015cisco}. Caches also play
an important role in the emergent \emph{Information Centric Networking} (ICN)
architecture, that incorporates them ubiquitously into the network in order to
increase its overall capacity~\cite{ahlgren2012survey}.

In order to improve network efficiency, cache servers are placed close to
the users, and store a subset of the \emph{catalog} of available documents. Upon
a user request for a document:
\begin{itemize}
  \item If the document is already stored in the cache, then the cache uploads it
    directly to the user. This event
    is called a \emph{cache hit}.

  \item Else, the request is forwarded to the repository server, which uploads a copy
    to the user, and possibly to the cache for future requests. This event is
    called a \emph{cache miss}.
\end{itemize}

Since the cost of storage is a constraint, each cache contains only a fraction
of the document catalog, and needs to eliminate some documents to free
space for new ones. Since the caches must decide to do so in real time, they use
simple distributed elimination algorithms, called \emph{cache eviction policies}.

Hereafter, we will focus our efforts on the \emph{Least Recently Used} (LRU) cache
eviction policy. 
To simplify the analysis, we will assume that all documents have the same size,
and therefore that the disk of the cache can be represented as a list of
documents of size $\CSz \geq 1$.
The LRU policy evicts content upon a user request as follows (see Fig.~\ref{fig:lru_cache}):
\begin{itemize}
  \item
    If the requested document is already stored in the cache, then it is moved to the front of the list,
    while all documents which were in front of it are shifted down by one slot.
  \item
    Else, a copy of the requested document is
    downloaded from the server and placed at the front of the list,
    while all documents which were already in the list are shifted down by one slot except the last one which is eliminated.
\end{itemize}
Intuitively, this simple policy should perform well, since highly requested documents
should stay near the front of the list, whereas unpopular ones should be
quickly eliminated. 

\begin{figure}[!thb]
  \centering
  \includegraphics[scale=0.33]{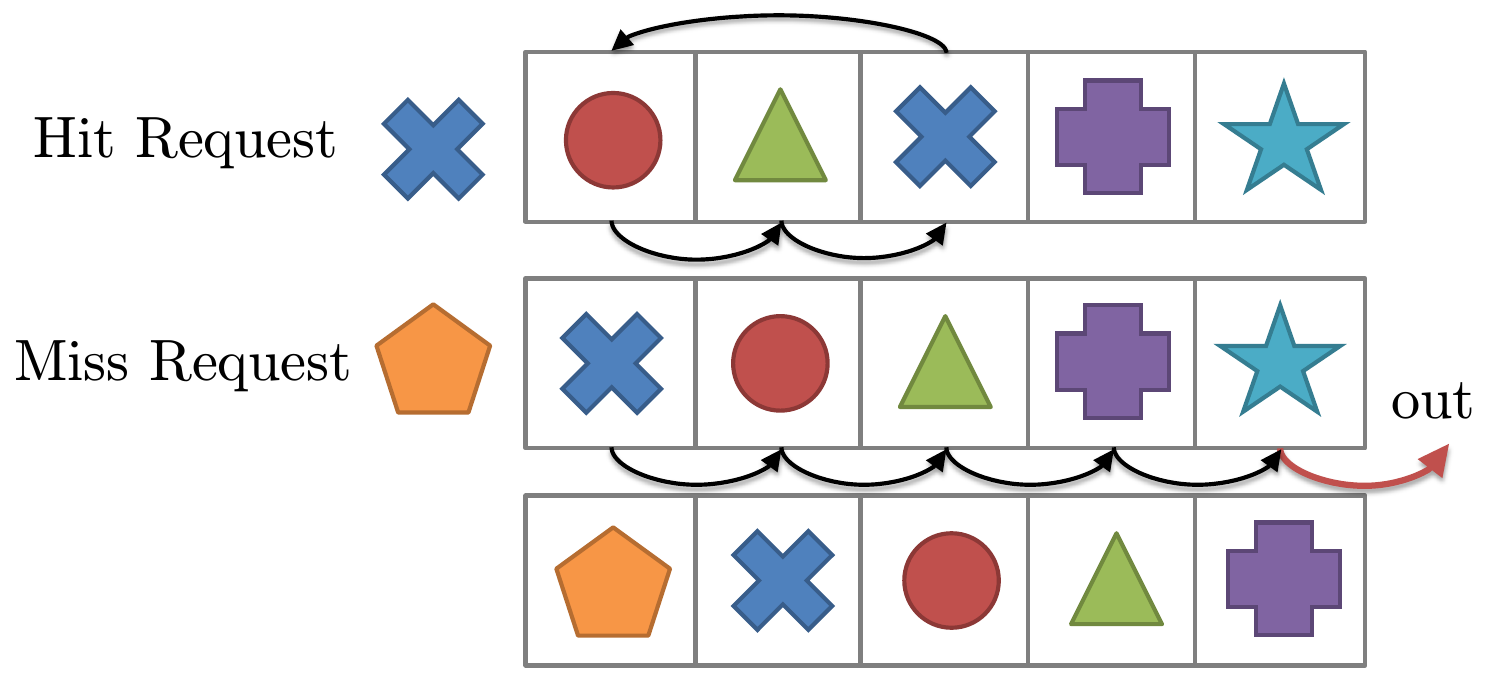}
  \caption{
    \emph{The LRU eviction policy handling a hit and a miss request on a cache
      of size $\CSz = 5$.}
  }
  \label{fig:lru_cache}
\end{figure}

Early theoretical studies on LRU caching performance further assumed that the
catalog is fixed and finite, and that documents there have each an intrinsic
probability to be requested independently, thus defining a popularity
distribution. The request process is then modeled as an \iid sequence, where at
each time step a document is requested according to its popularity. This
framework is commonly referred to as the \emph{Independent Reference Model}
(IRM) in the literature, see for instance~\cite{fricker2012versatile}.

While the IRM setting has been proved to be a good model for short time-scales,
it is not accurate for larger ones. In fact, other phenomena occurring within
longer time-scales must also be taken into account, notably the dynamic nature
of the catalog and of user preferences.

In order to capture these phenomena, a new model based on \emph{Poisson cluster
  point processes} has been independently proposed by Traverso et
al.~\cite{traverso2013temporal} and Olmos et al.~\cite{olmos2014catalog}. It
allows to address the catalog and preference dynamics, and thus to obtain more
accurate results in both large and small time scales. Its properties have 
received only heuristic analysis in these works.

The object of the present paper is to build a sound mathematical framework for
the analysis of this model, and to provide rigorous proofs for the estimation of
the hit probability, which corresponds to the asymptotic proportion of cache hits among all requests
when the number of these requests goes to infinity, or equivalently of the complementary miss probability. 

Before describing our main contributions, we briefly review the literature on
caching performance, mentioning only papers relevant to our present work;
see~\cite{fofack2012analysis} and the references therein for a more comprehensive
bibliography of the subject.
The modern treatment of the subject started with
Fill and Holst~\cite{fill1996distribution}, which introduced the embedding of
the request sequence into a marked Poisson process, in order to analyze the
related problem of the search cost for the \emph{Move-to-Front} list.
Independently, Che et al.~\cite{che2002hierarchical} also used marked Poisson
processes to model the requests. In their work, they express the hit
probability of a LRU cache in terms of a family of exit times of the documents
from the cache. In order to simplify the analysis, they approximated this
family by a single constant called the \emph{characteristic time}. This
heuristic, called the \emph{Che approximation} in the literature, proved to be
empirically accurate even outside of its original setting. The question of
quantifying the error incurred in the approximation has been partially answered
by Fricker et al.~\cite{fricker2012versatile}, where the authors provide a
justification for a Zipf popularity distribution when the cache size $\CSz$
grows to infinity and scales linearly with the catalog size. The error
incurred by the approximation is estimated for the exit times,
but not, however, for the hit probability.

In the present paper, we succeed in adapting the \emph{Che approximation} to the more complex setting of the cluster point process model.
The approximation accuracy has been considered first by Leonardi and
Torrisi~\cite{leonardi2015least}, which provide limit theorems for the exit time as $\CSz$ goes to infinity, 
as well as an upper bound of the
error on the hit probability. However, the latter bound depends on an
additional variable, of which the optimal value is not explicitly given in terms of
system parameters.

The contribution of our paper is threefold. Firstly, in Sections~\ref{sec:model}
and~\ref{sec:palm}, we use the Palm distribution for the system in order to
provide a probability space where an ``average document'' can be tagged and
analyzed independently from the rest. Secondly, in
Section~\ref{sec:hr_estimation},  we use the latter independence structure to
obtain an integral formula for the expected number of misses for a document,
generalizing the development in~\cite{olmos2014catalog}. Thirdly, in
Section~\ref{sec:asymptotic}, using scaling methods, we deduce from this formula an asymptotic
expansion for the average number of misses, showing that the error term 
in the \emph{Che approximation} is of order $O(1/\CSz)$. In contrast to the upper bound
provided in~\cite{leonardi2015least}, our error estimation depends simply on
system parameters and can be readily calculated. Section~\ref{sec:numerics} is
devoted to a numerical study validating the accuracy of the asymptotic
expansion. Section~\ref{sec:conclusion} contains some concluding remarks, and
Section~\ref{sec:proofs} contains all proofs.

\section{Document Request Model}
\label{sec:model}

Our request model consists in the following cluster point process on the real line $\RR$, illustrated
in Fig.~\ref{fig:request_process}. 

A ground process $\CatPP$, hereafter
called the \textbf{catalog arrival process}, gives the consecutive arrival times
of new documents to the catalog. We assume it to be a homogeneous Poisson
process with intensity $\CatArrRate>0$, and denote its generic arrival time by
$\CatArr$.

\begin{figure}[!thb]
  \centering
  \includegraphics[scale=0.4]{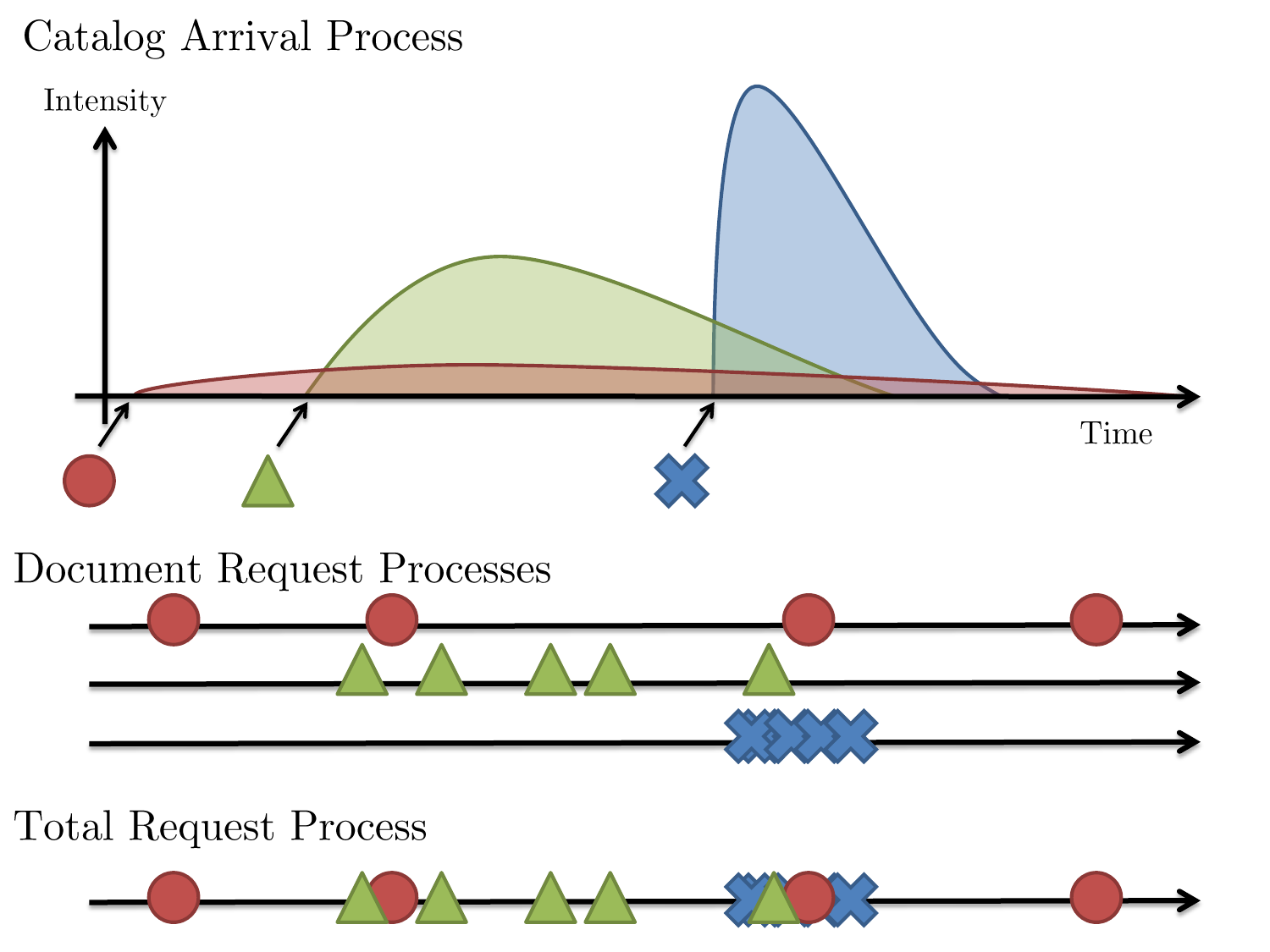}
  \caption{\textit{A sample of the document arrival and request processes.
      \textbf{Top}: Each catalog arrival triggers a function representing the request
      intensity for the corresponding document. \textbf{Bottom}: A sample of
      the document request processes. Their superposition generates the total
      request process.}}
  \label{fig:request_process}
\end{figure}

The cluster at an arrival time $\CatArr$ of $\CatPP$ is denoted by
$\DocReqPP_\CatArr$, and is an element of the space $\PPSpace{\RR}$ of point
processes on $\RR$.  It represents the \textbf{document request process} for
the document arriving to the catalog at that time~$a$. We assume that
$\DocReqPP_\CatArr$ is a Cox process directed by a stochastic intensity
function $\DocReqInt_\CatArr\ge0$ having the following properties.
\begin{itemize}
  \item Given $\CatPP$, the intensities $\DocReqInt_\CatArr$ for $\CatArr \in
    \CatPP$ are jointly independent.

  \item The intensities $\DocReqInt_\CatArr$ are \emph{causal}: each function
    $t \mapsto \DocReqInt_\CatArr(t)$ is zero for $t <\CatArr$.
    Requests
    for a document thus occur only after its arrival at the catalog.

  \item The distribution of $\DocReqInt_\CatArr$ is \emph{stationary}: for each
    arrival time $\CatArr \in \RR$, the processes $\DocReqInt_\CatArr(\cdot)$
    and $\DocReqInt_0(\cdot - \CatArr)$ have the same distribution.
\end{itemize}
These three conditions allow to sample the sequence
$(\DocReqInt_\CatArr)_{\CatArr \in \CatPP}$ using independent samples of a
\textbf{canonical intensity function} $\DocReqInt$ with support in $[0,
\infty)$, adequately shifted to every arrival time $\CatArr$.

For a document arriving at time $\CatArr$, we denote by $\DocReqMeas_\CatArr$
both the \textbf{mean function} associated to the request intensity
$\DocReqInt_\CatArr$ and the \textbf{average number of requests} (with abuse of
notation for conciseness)
\begin{equation*}
  \DocReqMeanFn_\CatArr(t)
  \defEqual
  \int_\CatArr^t \DocReqInt_\CatArr(u) \, \dd u\,,
  \quad
  t \geq \CatArr\,,
  \qquad
  \AveNReqs_\CatArr \defEqual \DocReqMeas_\CatArr(\infty)\,.
\end{equation*}
We assume that $\AveNReqs_\CatArr < \infty$ almost surely, and denote by
$\DocReqMeanComp_\CatArr$ the \textbf{complementary mean function}
\[
  \DocReqMeanComp_\CatArr(t)
  \defEqual \AveNReqs_\CatArr - \DocReqMeanFn_\CatArr(t)
  =
  \int_t^\infty \DocReqInt_\CatArr(u) \, \dd u\,, \quad t \geq \CatArr\,.
\]

When referring to the canonical document, which corresponds to an arrival at
time zero, we remove the time index $\CatArr$; for instance we write
$\AveNReqs$ and $\DocReqMeanComp(t)$. The superposition of all processes
$\DocReqPP_\CatArr$ for $\CatArr \in \CatPP$ given by
\[
  \TotReqPP \defEqual \sum_{\CatArr \in \CatPP} \DocReqPP_\CatArr
\]
constitutes the \textbf{total request process} for all documents. We assume
that
\begin{equation}
  \label{eq:finiteness_condition}
 \int_{-\infty}^t
  \E{1-e^{-(\DocReqMeanFn_\CatArr(t) - \DocReqMeanFn_\CatArr(s))}} \, \dd \CatArr
  <
  \infty
\end{equation}
for all times $s\leq t$. This is a necessary and sufficient condition for the process $\TotReqPP$ 
to be locally finite almost surely, see~\cite[Theorem
6.3.III]{daley2003introduction}.

\section{Tagging a Document via Palm Theory}
\label{sec:palm}

The key of our analysis is to \emph{tag} one document of the system and treat
the remaining process as an external \emph{environment}. To do this, we
follow~\cite[p.279]{daley2008introduction}. Let $\Palm{u, \nu}$ be the local
Palm distribution at point $(u, \nu)$ in $\RR \times \PPSpace{\RR}$ for the point process
\[
  \MarkCatPP \defEqual \sum_{\CatArr \in \CatPP} \delta_{\CatArr,\DocReqPP_\CatArr}
\]
constituted by the ground process $\CatPP$ marked with the document request
processes, which constitutes a Poisson point process on $\RR
\times\PPSpace{\RR}$. Define the mark-averaged Palm distribution $\APalm{u}$
on $\PPSpace{\RR}$ by
\[
  \APalm{u}(\cdot)
  \defEqual
  \E{\Palm{u, \DocReqPP_u}(\cdot)}.
\]
Under this distribution $\APalm{u}$, the process has the structure given by the
following proposition, illustrated by Fig.~\ref{fig:palm_schema}.

\begin{pro}[Palm Decomposition for Tagged Document]
  \label{pro:palm_distribution}
  Under the distribution $\APalm{u}$, the process $\tilde{\TotReqPP}$ has almost
  surely a point at time $u$. Furthermore:
  \begin{itemize}
    \item The distribution of the mark $\DocReqPP_u$ is the same as the original one.
    \item The distribution of the remaining process $\tilde{\TotReqPP}
      \setminus \delta_{u, \DocReqPP_u}$ is the same than that of the original
      process $\tilde{\TotReqPP}$.
    \item The mark $\DocReqPP_u$ and the process $\tilde{\TotReqPP} \setminus \delta_{u, \DocReqPP_u}$ are independent.
  \end{itemize}
\end{pro}
We refer to Section~\ref{sec:proof_palm_distribution} for the proof of this proposition.
\begin{figure}[!htb]
  \centering
  \includegraphics[scale=0.4]{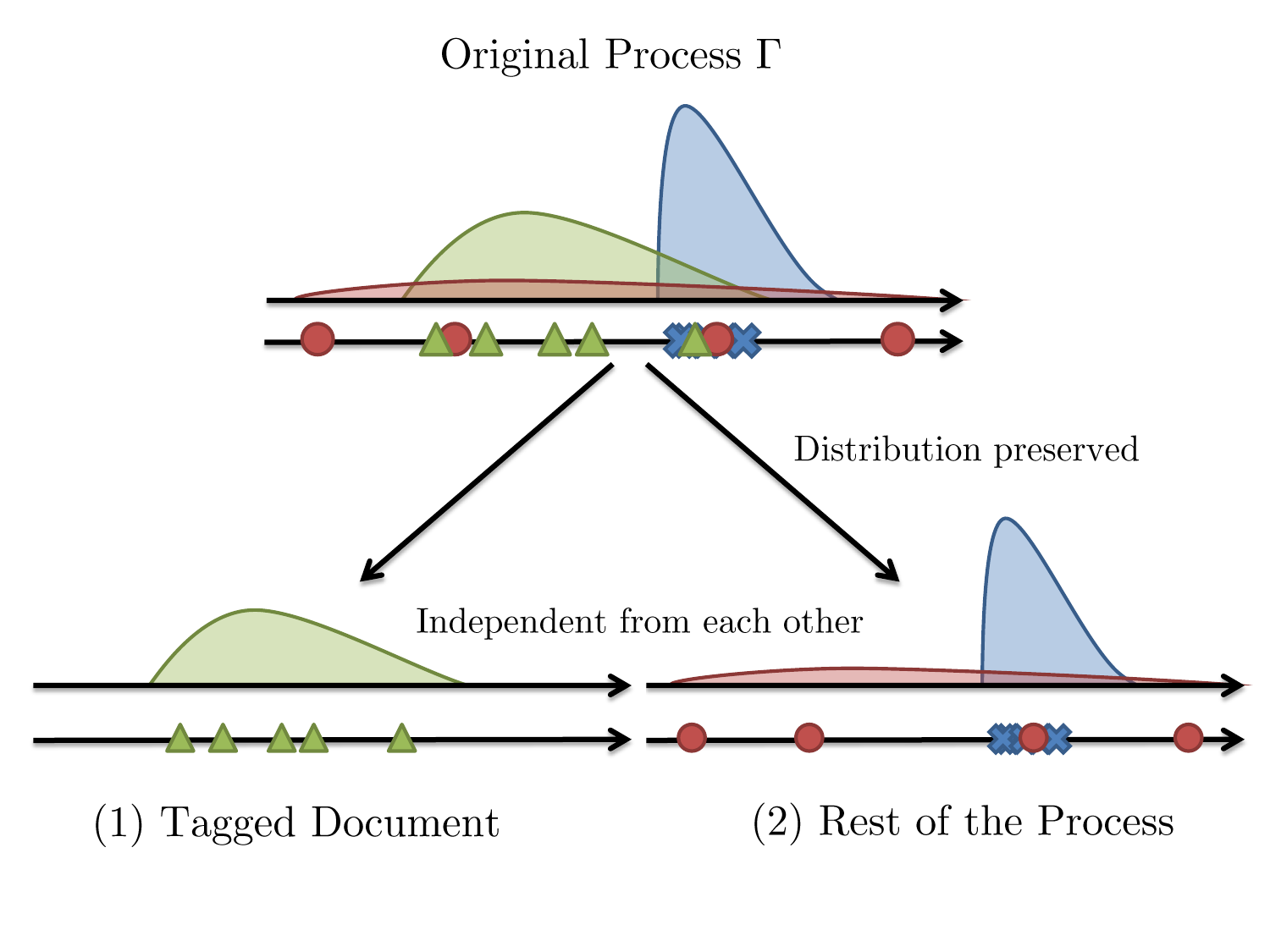}
  \caption{\textit{Illustration of request process $\DocReqPP$ under the averaged
      Palm distribution. The original process is decomposed into: (1) the tagged
      document and (2) the rest of the process. These are mutually independent, and
      the rest of the process has the same distribution as the original process.}}
  \label{fig:palm_schema}
\end{figure}

Proposition~\ref{pro:palm_distribution} allows us to consider a probability
space for which there is a document arrival at time $\CatArr = 0$, almost
surely. We call this document the \textbf{tagged} document, and the
complementary process \textbf{the rest}.

In the next section, we shall see that for the LRU caching discipline, the
independence of the tagged document from the rest allows us to derive a general
integral formula for the miss probability.
\section{Fundamental Integral Formula}
\label{sec:hr_estimation}
As stated in the previous section, we will consider a tagged document at time
zero, so that its associated distribution is the canonical one. For a LRU cache
with size $\CSz$, let $\NReqs$ and $\NMiss_\CSz$ be the random number of
requests and number of misses for the tagged document, respectively. The total
miss probability is defined by
\[
  \MissProba{\CSz} \defEqual \frac{\E{\NMiss_\CSz}}{\E{\NReqs}}\,,
\]
which is also the average per-document miss probability $\NMiss_\CSz/\NReqs$ under the
size biased distribution of $\NReqs$. 
The mixed Poisson random variable  $\NReqs$ with random mean $\AveNReqs$ has expectation
\[
  \E{\NReqs} = \E{\E{\NReqs}[\AveNReqs]} = \E{\AveNReqs},
\]
and it remains to study $\NMiss_\CSz$.

Let $(\Req_\iReq)_{\iReq=1}^\NReqs$ be the sequence of request times for the
tagged document, with the understanding that it is the empty set if $\NReqs =
0$. The first request being always a miss, the number of misses can be written
as
\begin{equation}
  \label{eq:nmiss}
  \NMiss_\CSz
  =
  \II{\NReqs \geq 1} + \II{\NReqs \geq 2}
  \sum_{\iReq = 2}^{\NReqs}
  \II{\text{Request at } \Req_\iReq \text { is a miss}
  }.
\end{equation}
Under the LRU policy, a document requested at time $s$ will be erased from the
cache at the first time, after the last request for this document,
that $\CSz$ distinct other documents have been requested.

For each $s$ in $\RR$, let us define the process $\NDifDocs^s = (\NDifDocs_t^s)_{t \geq s}$ which counts the number of \textbf{distinct} documents in the \textbf{rest} of
the process which are requested on the time interval $[s,t]$, and its \textbf{exit time} $\ExitTime{\CSz}^s$
to level $\CSz$.
Hence, $\ExitTime{\CSz}^s$  is the time that a document
requested at time $s$ spends in the cache before being evicted.
Denoting by
$F^s(\DocReqPP_\CatArr)$ the first arrival time of $\DocReqPP_\CatArr$ in $[s,
\infty)$, the process $\NDifDocs^s = (\NDifDocs_t^s)_{t \geq s}$ and exit time $\ExitTime{\CSz}^s$ can be
expressed as
\begin{equation}
  \label{eq:def_ndifdocs}
  \left\{
    \begin{aligned}
      \NDifDocs^s_t
      &\defEqual
      \#\{
      (\CatArr,\DocReqPP_\CatArr) \text{ in } \tilde{\TotReqPP}
      \setminus \delta_{0, \DocReqPP_0}: F^s(\DocReqPP_\CatArr) \leq t
      \}\,,
      \\
      \ExitTime{\CSz}^s &\defEqual \inf\{ t \geq s : \NDifDocs^s_t = \CSz \}\,.
    \end{aligned}
  \right.
\end{equation}
These definitions allow us to express the miss events as
\[
  \lbrace \text{Request at } \Req_{\iReq} \text{ is a miss}\rbrace
  =
  \left\lbrace X_{\Req_{\iReq}}^{\Req_{\iReq-1}} \geq \CSz \right\rbrace
  =
  \left\lbrace \Req_{\iReq} > \ExitTime{\CSz}^{\Req_{\iReq - 1}}\right\rbrace\,,
  \quad
  \iReq \geq 2\,,
\]
since such a miss occurs if and only if at least $\CSz$ distinct other
documents have been requested in the interval $[\Req_{\iReq-1} ,
\Req_{\iReq}]$.
Hence \eqref{eq:nmiss} can be written as
\begin{equation}
  \label{eq:nmiss_cech}
  \NMiss_\CSz
  =
  \II{\NReqs \geq 1} +
  \II{\NReqs \geq 2}
  \sum_{\iReq = 2}^{\NReqs}
  \II{\Req_\iReq > \ExitTime{\CSz}^{\Req_{\iReq-1}}
  }\,.
\end{equation}

To proceed further, we study the consequences of the
structure of the cluster point process on the
structure of the families $\NDifDocs^s$ and $\ExitTime{\CSz}^s$.

\begin{pro}[Characterization of $\NDifDocs^s$ and $\ExitTime{\CSz}^s$]
  \label{pro:characterization_ndifdocs}
  Let $s$ be in $\RR$.
  The process $\NDifDocs^s = (\NDifDocs_t^s)_{t \geq s}$ defined by~\eqref{eq:def_ndifdocs}
  is an inhomogeneous Poisson process with intensity function
  \begin{equation}
    \label{eq:ndifdocsmean}
    \NDifDocsMean^s(t)
    \defEqual
    \E{\NDifDocs^s_t}
    = \CatArrRate \, \int_{-\infty}^t
    \E{ 1 - e^{-(\DocReqMeanFn_\CatArr(t) - \DocReqMeanFn_\CatArr(s))}}
    \dd\CatArr
    \, ,
    \quad
    t \geq s\,
    .
  \end{equation}
  In particular, $\ExitTime{\CSz}^s - s \distEqual \ExitTime{\CSz}$, where
  $\ExitTime{\CSz} = \ExitTime{\CSz}^0$ is the exit time of a document
  requested at time zero.
\end{pro}

We refer to Section~\ref{sec:proof_characterization_ndifdocs} for the proof.

Equation~\eqref{eq:nmiss_cech},
Proposition~\ref{pro:characterization_ndifdocs},
and the independence between the tagged
document and the rest of the process now yield an integral formula for $\E{\NMiss_\CSz}$.

\begin{thm}[Integral Formula for Expected Misses]
\label{pro:expected_number_of_misses}
  
  The expected number of misses is given by
  \begin{equation}
    \label{eq:ave_nmiss_exit_time}
    \E{\NMiss_\CSz}
    =
    \E{\AveNMiss(\ExitTime{\CSz})}
  \end{equation}
  where $\ExitTime{\CSz} = \ExitTime{\CSz}^0$ denotes the exit time for a document requested at
  time zero, see \eqref{eq:def_ndifdocs}, and the function $\AveNMiss$ is defined 
  using the notation in Section~\ref{sec:model}  by
  \begin{equation}
    \label{eq:ave_nmiss_ttl}
    \AveNMiss(t)
    =
    \E{
      \int_0^{\infty} \!\!
      \DocReqInt(u)
      e^{-(\DocReqMeanFn(u+t) - \DocReqMeanFn(u))}
      \, \dd u
    },
    \quad
    t \geq 0\,.
  \end{equation}
  Moreover, $\lim_{t\to\infty} \downarrow m(t) = \MinAveNMiss$, where $\MinAveNMiss = \E{1 - e^{-\AveNReqs}}$
  with $\Lambda=\Lambda(\infty)$.
  \end{thm}

The proof is postponed to Section~\ref{sec:proof_expected_number_of_misses}.
It uses the following result of independent interest.

\begin{pro}[Functionals of Holding Times]
  \label{pro:holding_times_functionals}
  Let $\DocReqPP$ be an inhomogeneous Poisson process on $[0, \infty)$ with
  \emph{deterministic} intensity function $\DocReqInt$. Let the mean
  function $\DocReqMeanFn$ satisfy $\DocReqMeanFn(\infty) < \infty$, so that
  $\DocReqPP$ has a finite random number $\NReqs$ of points $(\Req_\iReq)_{\iReq = 1}^\NReqs$.
  Then, for any $F:\RR^+ \to \RR$,
  \begin{align*}
    &\E{
      \II{\NReqs \geq 2}
      \sum_{\iReq=2}^\NReqs F(\Req_\iReq - \Req_{\iReq -1})
    }
    \\
    &\quad = \int_0^\infty \dd w\,
    F(w) \int_0^\infty \dd u\,
    \DocReqInt(u) \DocReqInt(u+w)
    e^{-(\DocReqMeanFn(u+w) - \DocReqMeanFn(u))}
    \,\,.
  \end{align*}
\end{pro}

We refer to Section~\ref{sec:proof_holding_times_functionals} for the proof of this proposition.

The above analysis would identically apply if the random variable $\ExitTime{\CSz}$
were deterministic and equal to some positive constant $t$. This would
correspond to the cache discipline known as \emph{Time to Live (TTL)}, where
the cache evicts a document after a fixed amount of time $t$. Therefore,
$\AveNMiss(t)$ is simply the average number of misses for a TTL cache of
eviction time $t$. We can thus regard the number of misses in a LRU cache as a
time randomization of the misses in a TTL cache.

Indeed, the integral formula~\eqref{eq:ave_nmiss_ttl} in
Theorem~\ref{pro:expected_number_of_misses} can by rewritten using integration by parts as
\[
  m(t)
  =
  \E{
    \int_0^\infty
    \DocReqInt(u) \, e^{-(\DocReqMeanFn(u) - \DocReqMeanFn(u-t))} \, \dd u
  }
\]
which can be informally interpreted as follows: The exponential term
\[
  e^{-(\DocReqMeanFn(u) - \DocReqMeanFn(u-t))}
\]
is simply the conditional probability $\PP{\DocReqPP[u-t, u] =
  0}[][][\DocReqInt]$. Thus a request at time $u$ will contribute to the
intensity of the miss process if there were no requests in the interval $[u-t,
u]$, which is exactly a miss event in a $t$-TTL cache. This relationship
between the miss probabilities of TTL and LRU caches has been already noted by
Fofack et al.~\cite{fofack2014approximate}.

\section{Asymptotic Expansion}
\label{sec:asymptotic}

It holds that $\lim_{t\to\infty} \downarrow m(t) = \MinAveNMiss$, see
Theorem~\ref{pro:expected_number_of_misses}. 
Moreover, 
Proposition~\ref{pro:characterization_ndifdocs} yields that the exit time
$\ExitTime{\CSz}$ increases to infinity with $\CSz$. Hence,
\eqref{eq:ave_nmiss_exit_time} and dominated convergence yield that
\[
  \lim_{\CSz\to\infty}\E{\NMiss_\CSz} = \MinAveNMiss\,.
\]
This formula is not informative, since it basically tells us that the first request
for a document is the unique miss for infinite capacity.

A more interesting way to derive asymptotics for $\E{\NMiss_\CSz}$ is to scale some
system parameters with respect to $\CSz$. An intuitively good choice is to scale the catalog arrival rate
$\CatArrRate$ proportionally to the cache size $\CSz$.
In the following, with help of the results of the previous sections, we shall provide an asymptotic
expansion for $\E{\NMiss_\CSz}$ as $\CSz$ grows large in this scaling.

The canonical exit time $\ExitTime{\CSz}$ is the first passage time to
level $\CSz$ of an inhomogeneous Poisson process with mean function
$\NDifDocsMean = \NDifDocsMean^0$, see Proposition~\ref{pro:characterization_ndifdocs}. 
To pursue the analysis, we first prove a key
relation between $\NDifDocsMean$ and $\AveNMiss$.

\begin{pro}[Relation between $\NDifDocsMean$ and $\AveNMiss$]
  \label{pro:key_identity}
  The functions $\NDifDocsMean$ and $\AveNMiss$
  in \eqref{eq:ndifdocsmean} and \eqref{eq:ave_nmiss_ttl}
  satisfy $\NDifDocsMean'(t) = \CatArrRate \,\AveNMiss(t)$,
  and hence
  \[
    \NDifDocsMean(t) = \CatArrRate\, \IAveNMiss(t) \,,
    \;\;
    \text{where}
    \;\;
    \IAveNMiss(t) \defEqual \int_0^t \AveNMiss(s) \, \dd s\,,
    \qquad
    t \geq 0\,.
  \]
\end{pro}
We refer to Section~\ref{sec:proof_key_identity} for the proof.

Proposition~\ref{pro:key_identity} implies that 
$\NDifDocsMean(t) = y \Leftrightarrow \IAveNMiss(t) = y/ \CatArrRate$ 
and thus that
\begin{equation}
  \label{eq:key_identity_inverse}
  \NDifDocsMean^{-1}(y) = \IAveNMiss^{-1} \left( \frac{y}{\CatArrRate} \right),
  \quad y \geq 0\,,
\end{equation}
(for definiteness, we consider left-continuous inverses).
Moreover, the exit time $\ExitTime{\CSz}$ is the first passage time to level $\CSz$ of
an inhomogeneous Poisson process with mean function $\NDifDocsMean$
(see Proposition~\ref{pro:characterization_ndifdocs}), and  can be expressed as
\begin{equation}
  \label{eq:exit_time_change}
  \ExitTime{\CSz} = \NDifDocsMean^{-1}(\FPTPoiss{\CSz})
\end{equation}
where $\FPTPoiss{\CSz}$ is the first passage time to level $\CSz$ of an unit Poisson process and has a Gamma$(\CSz, 1)$
distribution. From Theorem~\ref{pro:expected_number_of_misses} and~\eqref{eq:exit_time_change},
we derive that
$
\E{\NMiss_\CSz}
=
\E{\AveNMiss(\ExitTime{\CSz})}
=
\E{\AveNMiss(\NDifDocsMean^{-1}(\FPTPoiss{\CSz}))}
$,
and \eqref{eq:key_identity_inverse} eventually yields that
\begin{equation}
  \label{eq:ave_nmiss_change_var2}
  \E{\NMiss_\CSz}
  =
  \E{
    \AveNMiss\left(\IAveNMiss^{-1}\left(\frac{\FPTPoiss{\CSz}}{\CatArrRate} \right)\right)
  }.
\end{equation}

Now, the strong law of large numbers yields that $\lim_{\CSz \to \infty} \FPTPoiss{\CSz}/\CSz = 1$
almost surely, and thus \eqref{eq:ave_nmiss_change_var2} strongly suggests to consider the scaling
\begin{equation}
  \label{eq:scaling}
  \CSz = \CatArrRate \AveSojourn
  \;\;\text{for some}\;\; \AveSojourn>0\,.
\end{equation}

This scaling is quite natural, since Little's law (\cite[Section 3.1.2]{baccelli2013elements}) applied to the cache system yields that
$\CSz = \CatArrRate \, \E{T^{\text{in}}_\CSz}$, where
\[
  T^{\text{in}}_\CSz = \int_0^\infty \II{ \text{Object is in the cache at } t} \, \dd t
\]
is the sojourn time of an object in the cache. Note that we do consider
the objects without any requests as entering the system, but
we set their sojourn time to $T_{\CSz}^{\text{in}} = 0$. 

As a consequence, the
asymptotic analysis under the scaling~\eqref{eq:scaling} amounts to fixing the
average sojourn time $\AveSojourn = \E{T_{\CSz}^{\text{in}}} = \CSz/\CatArrRate$
and the distribution of the canonical intensity function $\DocReqInt$ while
letting $\CSz$ grow to infinity.

Under the scaling~\eqref{eq:scaling}, eq.~\eqref{eq:ave_nmiss_change_var2} and
$\lim_{\CSz \to \infty} \FPTPoiss{\CSz}/\CSz = 1$ a.s. imply using dominated convergence that
\[
  \lim_{\CSz \to \infty} \E{\NMiss_\CSz}
  =
  \AveNMiss(\ChT{\AveSojourn})\,,
  \quad
  \ChT{\AveSojourn}\defEqual \IAveNMiss^{-1}(\AveSojourn)\,.
\]
In the following, the quantity $\ChT{\AveSojourn}$ will be called the
\textbf{characteristic time}. The asymptotics of $\E{\NMiss_\CSz}$ will be
expressed in terms of $\ChT{\AveSojourn}$. In this aim, we first recall two basic results
regarding the $\GammaD{\CSz}{1}$ distribution.
\begin{lem}[Classical Bounds on Gamma Laws]
  \label{lem:gamma}
  Let $\FPTPoiss{\CSz}$ follow a $\GammaD{\CSz}{1}$ distribution, and
  $X_\CSz \defEqual \FPTPoiss{\CSz}/\CSz$.
  Then:
  \begin{enumerate}[label=\textnormal{(\roman*)}, leftmargin=*, widest=ii]
    \item \label{lem:gamma-i}
    For any $\CSz > 1$ and $\eta > 0$,
      \[
        \PP{|X_\CSz - 1| \geq \eta} \leq 2e^{-\CSz \cdot \varphi(1 + \eta)}\,,
      \]
      where $\varphi(x) \defEqual x - 1 - \log x$ is the large deviations rate function for the law of large numbers for 
      exponential random variables of mean $1$.
    \item \label{lem:gamma-ii}
    For any $\CSz > 1$ and $k > 1$,
      \[
        \E{(X_\CSz - 1)^k} = O(\CSz^{- \lceil k/2 \rceil})\,.
      \]
  \end{enumerate}
\end{lem}
We refer to Section~\ref{sec:proof_gamma} for the classical proofs. We now
formulate our central result concerning the asymptotics for the average number
of misses.

\begin{thm}[Expected Number of Misses Expansion]
  \label{thm:asymptotic_expansion}
  Assume that the function $\AveNMiss$ is twice continuously differentiable in
  $(0, \infty)$. Let $\ChT{\AveSojourn}\defEqual
  \IAveNMiss^{-1}(\AveSojourn)$ (see Proposition~\ref{pro:key_identity}) and 
  \[
    \Error(\ChT{\AveSojourn})
    =
    \frac{\AveSojourn^2}{2\AveNMiss(\ChT{\AveSojourn})^2}
    \left(
      \AveNMiss''(\ChT{\AveSojourn})
      -
      \frac
      {\AveNMiss'(\ChT{\AveSojourn})^2}
      {\AveNMiss(\ChT{\AveSojourn})}
    \right).
  \]
  Then, as
  $\CSz$ goes to infinity with the scaling $\CSz = \CatArrRate \AveSojourn$ for
  fixed $\AveSojourn > 0$, we have
  \begin{equation}
    \label{eq:asymptotic_expansion}
    \E{\NMiss_\CSz}
    =
    \AveNMiss(\ChT{\AveSojourn})
    +
    \frac{\Error(\ChT{\AveSojourn})}{\CSz}
    +
    o \left( \rec{\CSz} \right).
  \end{equation} 
\end{thm}

We refer to Section~\ref{sec:asymptotic_expansion} for the proof.

Theorem~\ref{thm:asymptotic_expansion} justifies the accuracy
of the estimations that use the \emph{Che approximation}. In
the present setting, this heuristic consists in replacing the exit time
$\ExitTime{\CSz}$ in~\eqref{eq:ave_nmiss_exit_time} by the constant
$\widetilde{t}_\CSz = \NDifDocsMean^{-1}(\CSz)$, therefore estimating
$\E{\NMiss_\CSz}$ by $\AveNMiss(\widetilde{t}_\CSz)$. Now, under the scaling
$\CSz = \CatArrRate \AveSojourn$, the identity~\eqref{eq:key_identity_inverse}
entails that
\[
  \widetilde{t}_\CSz
  =
  \NDifDocsMean^{-1}(\CSz)
  =
  \IAveNMiss^{-1} \left( \frac{\CSz}{\CatArrRate} \right)
  =
  \IAveNMiss^{-1} (\AveSojourn)
  =
  \ChT{\AveSojourn}.
\]
The quantity $\widetilde{t}_\CSz$ is called in the literature the ``characteristic
time'', and this identity justifies this naming for
$\ChT{\AveSojourn}$ as well.
More importantly, the asymptotic expansion of
$\E{\NMiss_\CSz}$ in Theorem~\ref{thm:asymptotic_expansion} shows that the
error in the \emph{Che approximation} is of order $1/\CSz$ and specifies it precisely, for large $\CSz$ and
fixed average sojourn time $\AveSojourn$.

\begin{rmk}[Higher Order Expansions]
  \label{rmk:expansion}
  If the function $\AveNMiss$ has derivatives of higher order, the proof of
  Theorem~\ref{thm:asymptotic_expansion} together with Lemma~\ref{lem:gamma}
  allow us to derive higher order expansions of $\E{\NMiss_\CSz}$ in powers of
  $1/\CSz$. Specifically, to obtain an expansion at order $n$, we must expand
  $f_\AveSojourn$ to the $2n$-th order, since $\E{(X_\CSz -1)^k}$ is
  $O(1/\CSz^{\lceil k/2 \rceil})$ by Lemma~\ref{lem:gamma}.
  We then eventually obtain
  \[
    \E{\NMiss_\CSz} = \sum_{k=0}^{2n} \frac{f_\AveSojourn^{(k)}(1)}{j!} \frac{\phi_k(\CSz)}{\CSz^k} + o \left( \rec{\CSz^n} \right)
  \]
  where $\phi_k$ is a polynomial of degree $\lfloor k/2 \rfloor$, as shown in
  the proof of Lemma~\ref{lem:gamma}.
\end{rmk}

\begin{rmk}[Laplace Asymptotic Method]
  Theorem~\ref{thm:asymptotic_expansion} can be proved by purely
  analytical methods. Indeed, \eqref{eq:ave_nmiss_scaling} can be
  written in integral form, after using the change of variables $w \mapsto w/\CSz$, as
  \[
    \E{\NMiss_\CSz}
    =
    \frac{\CSz^\CSz}{\Gamma(\CSz)}
    \int_0^\infty e^{-\CSz(w - \log(w))} \frac{f_\AveSojourn(w)}{w} \, \dd w\,.
  \]
  Theorem~\ref{thm:asymptotic_expansion} then follows by expanding this integral
  using the Laplace method (see~\cite[(3.15)]{miller2006applied}) and
  $\Gamma(\CSz)$ using the Stirling formula. The expansion of the numerator must
  be performed through a Taylor series of function $f_\AveSojourn$ around the
  extremal point of the argument $w - \log (w)$ of the exponential term, that
  is, near $w=1$. This method is, however, more complicated, since it involves
  the expansion of both numerator and denominator in powers of $\sqrt{\CSz}$.
\end{rmk}

The smoothness assumptions on the function $\AveNMiss$ in
Theorem~\ref{thm:asymptotic_expansion} can usually be checked readily on a case
by case basis, by justifying interchange of derivation and expectation in
\eqref{eq:ave_nmiss_ttl} using dominated convergence. Nevertheless, it is
difficult to give a general result.

To conclude this section, we show that these smoothness assumptions hold for a
class of random intensities $\DocReqInt$ which is suitable for modeling
purposes. This class is built by randomly scaling a deterministic shape function
in both domain and range. It includes the families used in previous
works~\cite{traverso2013temporal,olmos2014catalog}.

\begin{pro}[Twice Continuously Differentiable Example]
  \label{pro:scale_family}
  Let $f \in \mathcal{C}^1(0, \infty)$ be a strictly positive unimodal function
  satisfying that $\int f = 1$, $\int f^2 < \infty$, and $\int |f'| < \infty$.
  Let
  $(\Pop, \Lifespan)$ be a couple of positive random variables with a smooth joint
  density, satisfying that $\E{\Pop} < \infty$ and $\E{\Pop \Lifespan} < \infty$. If
  the canonical document request intensity is of the form
  \begin{equation}
    \label{eq:scale_family_def}
    \DocReqInt(u) = \Pop \cdot f \! \left( \frac{u}{\Lifespan} \right), \quad u \geq 0,
  \end{equation}
  then the function $\AveNMiss$ is $\mathcal{C}^2(0, \infty)$ with derivatives
  given for $t > 0$ by
  \begin{equation}
    \label{eq:derivatives}
    \left\{
      \begin{aligned}
        \AveNMiss'(t)
        &=
        -\E{
          \Pop^2 \Lifespan
          \int_0^\infty
          f(u) f\!\left(u + \frac{t}{\Lifespan} \right)
          e^{-\Pop \Lifespan
            \left( F \left(u+\frac{t}{\Lifespan}\right) - F(u) \right)
          } \, \dd u
        },
        \\
        \AveNMiss''(t)
        &=
        \E{
          \Pop^3 \Lifespan
          \int_0^\infty
          f(u)
          f \!\left( u + \frac{t}{\Lifespan} \right)^2
          e^{-\Pop \Lifespan
            \left( F \left(u+\frac{t}{\Lifespan}\right) - F(u) \right)
          }
          \, \dd u
        }
        \\
        &\quad-
        \E{
          \Pop^2
          \int_0^\infty
          f(u)
          f' \!\left(u + \frac{t}{\Lifespan} \right)
          e^{-\Pop \Lifespan
            \left( F \left(u+\frac{t}{\Lifespan}\right) - F(u) \right)
          }
          \, \dd u
        },
      \end{aligned}
    \right.
  \end{equation}
  where $F(u) = \int_0^u f(v) dv$.
\end{pro}
We defer the proof of the  proposition to Section~\ref{sec:differentiability}.

Note that Proposition~\ref{pro:scale_family} only imposes mild conditions on
the distribution of $(\Pop, \Lifespan)$. The admitted shape functions $f$
include exponential and power law decreasing profiles, and
Gaussian curves restricted to $[0, \infty)$. In addition, the assumption of $f$
being strictly positive on $[0, \infty)$ can be weakened to that of being
positive only in a compact interval; this in turn implies that $f'$ is not
differentiable everywhere and the second derivative of $\AveNMiss$ will thus
contain additional terms from the integral of $f'$. These terms can be obtained
by integration by parts (see~\cite[Th. 3.36]{folland2013real} for a
generalized form).


One example of such a family with compact support is given by the ``Box Model'',
previously analyzed in~\cite{olmos2014catalog}, which can be constructed by
simply taking $f = \I_{[0,1]}$. In this case, $\AveNMiss$ and its derivatives
reduce to
\begin{equation}
  \label{eq:derivatives_box}
  \left\{
    \begin{aligned}
      \AveNMiss(t)
      &=
      \E{
        \left( 1-e^{-\Pop \Lifespan} \right) \I_{\{\Lifespan \leq t\}}
        +
        \left( 1-e^{-\Pop t} + \Pop (\Lifespan - t) e^{-\Pop t} \right) \I_{\{\Lifespan > t\}}
      },
      \\
      \AveNMiss'(t)
      &=
      -\E{\Pop^2(\Lifespan - t) e^{-\Pop t} \, \I_{\{\Lifespan > t\}}},
      \\
      \AveNMiss''(t)
      &=
      \E{(\Pop^2 + \Pop^3(\Lifespan - t))e^{-\Pop t}\, \I_{\{\Lifespan > t\}}}.
    \end{aligned}
  \right.
\end{equation}
We will use this model for a numerical illustration in the next section.

\section{Numerical Experiments}
\label{sec:numerics}
We provide some numerical results to validate the accuracy of asymptotic
expansion~\eqref{eq:asymptotic_expansion}, by comparing it to the values
obtained from the system simulation. In our experiments, we used the ``Box
Model'' in which the canonical intensity function is given by
\[
  \DocReqInt(u) = \Pop \cdot \II{ 0 \leq u \leq \Lifespan}, \quad u \geq 0,
\]
where the random pair $(\Pop, \Lifespan)$
represents the request rate and lifespan of a document. In view
of~\eqref{eq:asymptotic_expansion}, we obtain the zero order and first order
approximations for the hit probability $\HitProba{\CSz}$, namely
\begin{equation}
  \label{eq:hit_ratio_approx}
  \HitProba{\CSz}
  =
  1 - \MissProba{\CSz}
  =
  1 - \frac{\E{\NMiss_\CSz}}{\E{\AveNReqs}}
  \approx
  \begin{dcases}
    1
    -
    \frac{\AveNMiss(\ChT{\AveSojourn})}{\E{\AveNReqs}},
    & 0\text{-th Order}
    \\
    1
    -
    \frac{\AveNMiss(\ChT{\AveSojourn}) + e(\ChT{\AveSojourn})/\CSz}{\E{\AveNReqs}},
    & 1\text{-st Order}
  \end{dcases}
\end{equation}
where $\E{\AveNReqs} = \E{\Pop \Lifespan}$.

For a given general distribution of $(\Pop, \Lifespan)$, we cannot deduce explicit expressions for $\AveNMiss,
\AveNMiss', \AveNMiss'', \IAveNMiss$, and $\IAveNMiss^{-1}$ from~\eqref{eq:derivatives_box}. In particular,
there is usually no formula for $\ChT{\AveSojourn}$ in terms of
$\AveSojourn$. In consequence, we resorted to numerical integration and inversion
to obtain the hit probability estimates in~\eqref{eq:hit_ratio_approx}.

As argued in~\cite{olmos2014catalog}, actual data traces suggest that the
distributions of variable $\Pop$ and $\Lifespan$ are heavy tailed with infinite
variance, that is, with tail index $\alpha \in (1,2)$. For our experiments, we
consequently chose $\Pop$ and $\Lifespan$ to be distributed as independent
Pareto-Lomax variables, with probability density $\alpha \sigma^\alpha/(\sigma
+ x)^{\alpha + 1}$ for $ x > 0$, with respective parameters $(\alpha=1.9,
\sigma=22.5)$ and $(\alpha=1.7, \sigma=0.07)$. Such values have been taken so
that the simulation time is not excessive; they provide a ``box'' of average
width $0.1$ and height $25$ with high volatility since neither $\Pop$ nor
$\Lifespan$ have a finite variance.

We generated the request process associated with these intensity functions for
various values of $\CatArrRate$ ranging from $10$ to $1{,}000$. For each request
sequence, we simulated an LRU cache and obtained the empirical hit probability
for various capacities $\CSz$.

To obtain reliable results, the heavy tailed
nature of the input distributions requires to use the stable-law central limit
theorem (see \cite[Th. 4.5.1]{whitt2002stochastic}). Specifically, there
exists a so-called stable law $S_\alpha(\sigma, \beta, \mu)$ with scaling
parameter $\sigma$ and a constant $K_\alpha$ such that, in distribution,
\[
  \lim_{n \to \infty} \rec{K_\alpha} \rec{n^{1/\alpha}} \sum_{i=1}^n (\Lifespan_i - n \E{\Lifespan})
  =
  S_\alpha(1, 1, 0)\,.
\]
This allows to heuristically quantify the convergence
rate for the law of large numbers by considering that
\[
  \rec{n} \sum_{i=1}^n (\Lifespan_i - n \E{\Lifespan})
  \underset{n \to \infty}{\approx}
  S_{\alpha}\left(\frac{K_{\alpha}}{n^{1-1/\alpha}}, 1, 0 \right)
\]
(in the present case, $\alpha=1.7$ for $\Lifespan$). We
then chose the simulation time $S$ such that the average number of observed documents $n =
\CatArrRate S \times \E{1 - e^{-\Pop \Lifespan}}$ is such that scaling parameter
$K_{\alpha}/n^{1-1/\alpha}$ is smaller than $10^{-3}$ (such a value of $n$ ensures the same accuracy
for the request rate $\Pop$ with larger tail index $\alpha=1.9$).
Besides, we also chose $S$ large enough to ensure that there is enough time for all observable
documents to appear in the simulated trace.

We show in Fig.~\ref{fig:convergence} some of the resulting hit probability curves from these
experiments. We observe that the zero order approximation in~\eqref{eq:hit_ratio_approx} is almost exact
already for $\CatArrRate = 500$. The error incurred by the approximation for lower $\CatArrRate$ can
be corrected by using the first order approximation in~\eqref{eq:hit_ratio_approx}, as shown in
Fig.~\ref{fig:small_rate} for $\CatArrRate = 50$. For even lower intensities, this correction
might not be enough to approximate the real hit probability, as illustrated in
Fig.~\ref{fig:smaller_rate} for $\CatArrRate = 10$; the higher order expansion of
Remark~\ref{rmk:expansion} would then be needed.

\begin{figure}[!thb]
\begingroup
  \makeatletter
  \providecommand\color[2][]{%
    \GenericError{(gnuplot) \space\space\space\@spaces}{%
      Package color not loaded in conjunction with
      terminal option `colourtext'%
    }{See the gnuplot documentation for explanation.%
    }{Either use 'blacktext' in gnuplot or load the package
      color.sty in LaTeX.}%
    \renewcommand\color[2][]{}%
  }%
  \providecommand\includegraphics[2][]{%
    \GenericError{(gnuplot) \space\space\space\@spaces}{%
      Package graphicx or graphics not loaded%
    }{See the gnuplot documentation for explanation.%
    }{The gnuplot epslatex terminal needs graphicx.sty or graphics.sty.}%
    \renewcommand\includegraphics[2][]{}%
  }%
  \providecommand\rotatebox[2]{#2}%
  \@ifundefined{ifGPcolor}{%
    \newif\ifGPcolor
    \GPcolortrue
  }{}%
  \@ifundefined{ifGPblacktext}{%
    \newif\ifGPblacktext
    \GPblacktexttrue
  }{}%
  \let\gplgaddtomacro\g@addto@macro
  \gdef\gplbacktext{}%
  \gdef\gplfronttext{}%
  \makeatother
  \ifGPblacktext
    \def\colorrgb#1{}%
    \def\colorgray#1{}%
  \else
    \ifGPcolor
      \def\colorrgb#1{\color[rgb]{#1}}%
      \def\colorgray#1{\color[gray]{#1}}%
      \expandafter\def\csname LTw\endcsname{\color{white}}%
      \expandafter\def\csname LTb\endcsname{\color{black}}%
      \expandafter\def\csname LTa\endcsname{\color{black}}%
      \expandafter\def\csname LT0\endcsname{\color[rgb]{1,0,0}}%
      \expandafter\def\csname LT1\endcsname{\color[rgb]{0,1,0}}%
      \expandafter\def\csname LT2\endcsname{\color[rgb]{0,0,1}}%
      \expandafter\def\csname LT3\endcsname{\color[rgb]{1,0,1}}%
      \expandafter\def\csname LT4\endcsname{\color[rgb]{0,1,1}}%
      \expandafter\def\csname LT5\endcsname{\color[rgb]{1,1,0}}%
      \expandafter\def\csname LT6\endcsname{\color[rgb]{0,0,0}}%
      \expandafter\def\csname LT7\endcsname{\color[rgb]{1,0.3,0}}%
      \expandafter\def\csname LT8\endcsname{\color[rgb]{0.5,0.5,0.5}}%
    \else
      \def\colorrgb#1{\color{black}}%
      \def\colorgray#1{\color[gray]{#1}}%
      \expandafter\def\csname LTw\endcsname{\color{white}}%
      \expandafter\def\csname LTb\endcsname{\color{black}}%
      \expandafter\def\csname LTa\endcsname{\color{black}}%
      \expandafter\def\csname LT0\endcsname{\color{black}}%
      \expandafter\def\csname LT1\endcsname{\color{black}}%
      \expandafter\def\csname LT2\endcsname{\color{black}}%
      \expandafter\def\csname LT3\endcsname{\color{black}}%
      \expandafter\def\csname LT4\endcsname{\color{black}}%
      \expandafter\def\csname LT5\endcsname{\color{black}}%
      \expandafter\def\csname LT6\endcsname{\color{black}}%
      \expandafter\def\csname LT7\endcsname{\color{black}}%
      \expandafter\def\csname LT8\endcsname{\color{black}}%
    \fi
  \fi
  \setlength{\unitlength}{0.0500bp}%
  \begin{picture}(5760.00,3556.00)%
    \gplgaddtomacro\gplbacktext{%
      \csname LTb\endcsname%
      \put(660,640){\makebox(0,0)[r]{\strut{} 0}}%
      \csname LTb\endcsname%
      \put(660,1175){\makebox(0,0)[r]{\strut{} 0.2}}%
      \csname LTb\endcsname%
      \put(660,1710){\makebox(0,0)[r]{\strut{} 0.4}}%
      \csname LTb\endcsname%
      \put(660,2245){\makebox(0,0)[r]{\strut{} 0.6}}%
      \csname LTb\endcsname%
      \put(660,2780){\makebox(0,0)[r]{\strut{} 0.8}}%
      \csname LTb\endcsname%
      \put(660,3315){\makebox(0,0)[r]{\strut{} 1}}%
      \csname LTb\endcsname%
      \put(780,440){\makebox(0,0){\strut{} 0}}%
      \csname LTb\endcsname%
      \put(1704,440){\makebox(0,0){\strut{} 0.05}}%
      \csname LTb\endcsname%
      \put(2628,440){\makebox(0,0){\strut{} 0.1}}%
      \csname LTb\endcsname%
      \put(3551,440){\makebox(0,0){\strut{} 0.15}}%
      \csname LTb\endcsname%
      \put(4475,440){\makebox(0,0){\strut{} 0.2}}%
      \csname LTb\endcsname%
      \put(5399,440){\makebox(0,0){\strut{} 0.25}}%
      \put(200,1977){\rotatebox{-270}{\makebox(0,0){\strut{}Hit Probability}}}%
      \put(3089,140){\makebox(0,0){\strut{}Mean Sojourn Time $\theta = C/\gamma$}}%
    }%
    \gplgaddtomacro\gplfronttext{%
      \csname LTb\endcsname%
      \put(4736,1628){\makebox(0,0)[r]{\strut{}Simulation for $\gamma = 12$}}%
      \csname LTb\endcsname%
      \put(4736,1378){\makebox(0,0)[r]{\strut{}Simulation for $\gamma = 50$}}%
      \csname LTb\endcsname%
      \put(4736,1128){\makebox(0,0)[r]{\strut{}Simulation for $\gamma = 500$}}%
      \csname LTb\endcsname%
      \put(4736,878){\makebox(0,0)[r]{\strut{}0th Order Approximation}}%
    }%
    \gplbacktext
    \put(0,0){\includegraphics{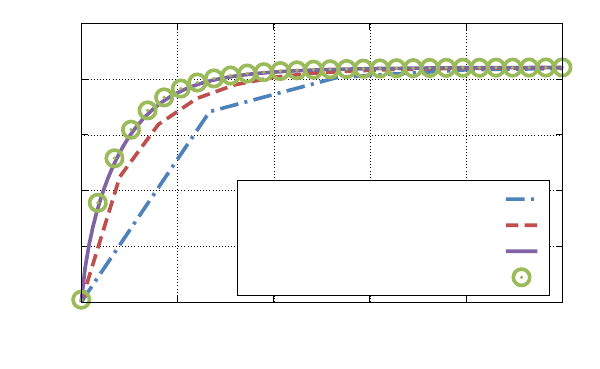}}%
    \gplfronttext
  \end{picture}%
\endgroup
  \caption{\textit{Convergence of the hit probability curves obtained in the experiments to the 0th
      order Che approximation.}}
  \label{fig:convergence}
\end{figure}

\begin{figure}[!thb]
  \begin{subfigure}{\textwidth}
    \centering
\begingroup
  \makeatletter
  \providecommand\color[2][]{%
    \GenericError{(gnuplot) \space\space\space\@spaces}{%
      Package color not loaded in conjunction with
      terminal option `colourtext'%
    }{See the gnuplot documentation for explanation.%
    }{Either use 'blacktext' in gnuplot or load the package
      color.sty in LaTeX.}%
    \renewcommand\color[2][]{}%
  }%
  \providecommand\includegraphics[2][]{%
    \GenericError{(gnuplot) \space\space\space\@spaces}{%
      Package graphicx or graphics not loaded%
    }{See the gnuplot documentation for explanation.%
    }{The gnuplot epslatex terminal needs graphicx.sty or graphics.sty.}%
    \renewcommand\includegraphics[2][]{}%
  }%
  \providecommand\rotatebox[2]{#2}%
  \@ifundefined{ifGPcolor}{%
    \newif\ifGPcolor
    \GPcolortrue
  }{}%
  \@ifundefined{ifGPblacktext}{%
    \newif\ifGPblacktext
    \GPblacktexttrue
  }{}%
  \let\gplgaddtomacro\g@addto@macro
  \gdef\gplbacktext{}%
  \gdef\gplfronttext{}%
  \makeatother
  \ifGPblacktext
    \def\colorrgb#1{}%
    \def\colorgray#1{}%
  \else
    \ifGPcolor
      \def\colorrgb#1{\color[rgb]{#1}}%
      \def\colorgray#1{\color[gray]{#1}}%
      \expandafter\def\csname LTw\endcsname{\color{white}}%
      \expandafter\def\csname LTb\endcsname{\color{black}}%
      \expandafter\def\csname LTa\endcsname{\color{black}}%
      \expandafter\def\csname LT0\endcsname{\color[rgb]{1,0,0}}%
      \expandafter\def\csname LT1\endcsname{\color[rgb]{0,1,0}}%
      \expandafter\def\csname LT2\endcsname{\color[rgb]{0,0,1}}%
      \expandafter\def\csname LT3\endcsname{\color[rgb]{1,0,1}}%
      \expandafter\def\csname LT4\endcsname{\color[rgb]{0,1,1}}%
      \expandafter\def\csname LT5\endcsname{\color[rgb]{1,1,0}}%
      \expandafter\def\csname LT6\endcsname{\color[rgb]{0,0,0}}%
      \expandafter\def\csname LT7\endcsname{\color[rgb]{1,0.3,0}}%
      \expandafter\def\csname LT8\endcsname{\color[rgb]{0.5,0.5,0.5}}%
    \else
      \def\colorrgb#1{\color{black}}%
      \def\colorgray#1{\color[gray]{#1}}%
      \expandafter\def\csname LTw\endcsname{\color{white}}%
      \expandafter\def\csname LTb\endcsname{\color{black}}%
      \expandafter\def\csname LTa\endcsname{\color{black}}%
      \expandafter\def\csname LT0\endcsname{\color{black}}%
      \expandafter\def\csname LT1\endcsname{\color{black}}%
      \expandafter\def\csname LT2\endcsname{\color{black}}%
      \expandafter\def\csname LT3\endcsname{\color{black}}%
      \expandafter\def\csname LT4\endcsname{\color{black}}%
      \expandafter\def\csname LT5\endcsname{\color{black}}%
      \expandafter\def\csname LT6\endcsname{\color{black}}%
      \expandafter\def\csname LT7\endcsname{\color{black}}%
      \expandafter\def\csname LT8\endcsname{\color{black}}%
    \fi
  \fi
  \setlength{\unitlength}{0.0500bp}%
  \begin{picture}(5182.00,3174.00)%
    \gplgaddtomacro\gplbacktext{%
      \csname LTb\endcsname%
      \put(660,640){\makebox(0,0)[r]{\strut{} 0}}%
      \csname LTb\endcsname%
      \put(660,1099){\makebox(0,0)[r]{\strut{} 0.2}}%
      \csname LTb\endcsname%
      \put(660,1557){\makebox(0,0)[r]{\strut{} 0.4}}%
      \csname LTb\endcsname%
      \put(660,2016){\makebox(0,0)[r]{\strut{} 0.6}}%
      \csname LTb\endcsname%
      \put(660,2474){\makebox(0,0)[r]{\strut{} 0.8}}%
      \csname LTb\endcsname%
      \put(660,2933){\makebox(0,0)[r]{\strut{} 1}}%
      \csname LTb\endcsname%
      \put(780,440){\makebox(0,0){\strut{} 0}}%
      \csname LTb\endcsname%
      \put(1588,440){\makebox(0,0){\strut{} 2}}%
      \csname LTb\endcsname%
      \put(2396,440){\makebox(0,0){\strut{} 4}}%
      \csname LTb\endcsname%
      \put(3205,440){\makebox(0,0){\strut{} 6}}%
      \csname LTb\endcsname%
      \put(4013,440){\makebox(0,0){\strut{} 8}}%
      \csname LTb\endcsname%
      \put(4821,440){\makebox(0,0){\strut{} 10}}%
      \put(200,1786){\rotatebox{-270}{\makebox(0,0){\strut{}Hit Probability}}}%
      \put(2800,140){\makebox(0,0){\strut{}Cache Size}}%
    }%
    \gplgaddtomacro\gplfronttext{%
      \csname LTb\endcsname%
      \put(4158,1378){\makebox(0,0)[r]{\strut{}Simulation for $\gamma = 50$}}%
      \csname LTb\endcsname%
      \put(4158,1128){\makebox(0,0)[r]{\strut{}0th order approximation}}%
      \csname LTb\endcsname%
      \put(4158,878){\makebox(0,0)[r]{\strut{}1st order approximation}}%
    }%
    \gplbacktext
    \put(0,0){\includegraphics{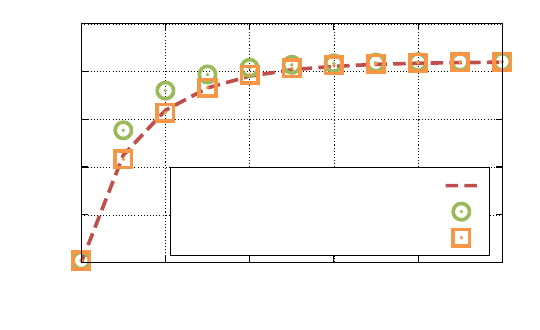}}%
    \gplfronttext
  \end{picture}%
\endgroup
    \caption{\textit{Approximations for $\CatArrRate = 50$}}
    \label{fig:small_rate}
  \end{subfigure}
  \begin{subfigure}{\textwidth}
    \centering
\begingroup
  \makeatletter
  \providecommand\color[2][]{%
    \GenericError{(gnuplot) \space\space\space\@spaces}{%
      Package color not loaded in conjunction with
      terminal option `colourtext'%
    }{See the gnuplot documentation for explanation.%
    }{Either use 'blacktext' in gnuplot or load the package
      color.sty in LaTeX.}%
    \renewcommand\color[2][]{}%
  }%
  \providecommand\includegraphics[2][]{%
    \GenericError{(gnuplot) \space\space\space\@spaces}{%
      Package graphicx or graphics not loaded%
    }{See the gnuplot documentation for explanation.%
    }{The gnuplot epslatex terminal needs graphicx.sty or graphics.sty.}%
    \renewcommand\includegraphics[2][]{}%
  }%
  \providecommand\rotatebox[2]{#2}%
  \@ifundefined{ifGPcolor}{%
    \newif\ifGPcolor
    \GPcolortrue
  }{}%
  \@ifundefined{ifGPblacktext}{%
    \newif\ifGPblacktext
    \GPblacktexttrue
  }{}%
  \let\gplgaddtomacro\g@addto@macro
  \gdef\gplbacktext{}%
  \gdef\gplfronttext{}%
  \makeatother
  \ifGPblacktext
    \def\colorrgb#1{}%
    \def\colorgray#1{}%
  \else
    \ifGPcolor
      \def\colorrgb#1{\color[rgb]{#1}}%
      \def\colorgray#1{\color[gray]{#1}}%
      \expandafter\def\csname LTw\endcsname{\color{white}}%
      \expandafter\def\csname LTb\endcsname{\color{black}}%
      \expandafter\def\csname LTa\endcsname{\color{black}}%
      \expandafter\def\csname LT0\endcsname{\color[rgb]{1,0,0}}%
      \expandafter\def\csname LT1\endcsname{\color[rgb]{0,1,0}}%
      \expandafter\def\csname LT2\endcsname{\color[rgb]{0,0,1}}%
      \expandafter\def\csname LT3\endcsname{\color[rgb]{1,0,1}}%
      \expandafter\def\csname LT4\endcsname{\color[rgb]{0,1,1}}%
      \expandafter\def\csname LT5\endcsname{\color[rgb]{1,1,0}}%
      \expandafter\def\csname LT6\endcsname{\color[rgb]{0,0,0}}%
      \expandafter\def\csname LT7\endcsname{\color[rgb]{1,0.3,0}}%
      \expandafter\def\csname LT8\endcsname{\color[rgb]{0.5,0.5,0.5}}%
    \else
      \def\colorrgb#1{\color{black}}%
      \def\colorgray#1{\color[gray]{#1}}%
      \expandafter\def\csname LTw\endcsname{\color{white}}%
      \expandafter\def\csname LTb\endcsname{\color{black}}%
      \expandafter\def\csname LTa\endcsname{\color{black}}%
      \expandafter\def\csname LT0\endcsname{\color{black}}%
      \expandafter\def\csname LT1\endcsname{\color{black}}%
      \expandafter\def\csname LT2\endcsname{\color{black}}%
      \expandafter\def\csname LT3\endcsname{\color{black}}%
      \expandafter\def\csname LT4\endcsname{\color{black}}%
      \expandafter\def\csname LT5\endcsname{\color{black}}%
      \expandafter\def\csname LT6\endcsname{\color{black}}%
      \expandafter\def\csname LT7\endcsname{\color{black}}%
      \expandafter\def\csname LT8\endcsname{\color{black}}%
    \fi
  \fi
  \setlength{\unitlength}{0.0500bp}%
  \begin{picture}(5182.00,3174.00)%
    \gplgaddtomacro\gplbacktext{%
      \csname LTb\endcsname%
      \put(660,640){\makebox(0,0)[r]{\strut{} 0}}%
      \csname LTb\endcsname%
      \put(660,1099){\makebox(0,0)[r]{\strut{} 0.2}}%
      \csname LTb\endcsname%
      \put(660,1557){\makebox(0,0)[r]{\strut{} 0.4}}%
      \csname LTb\endcsname%
      \put(660,2016){\makebox(0,0)[r]{\strut{} 0.6}}%
      \csname LTb\endcsname%
      \put(660,2474){\makebox(0,0)[r]{\strut{} 0.8}}%
      \csname LTb\endcsname%
      \put(660,2933){\makebox(0,0)[r]{\strut{} 1}}%
      \csname LTb\endcsname%
      \put(780,440){\makebox(0,0){\strut{} 0}}%
      \csname LTb\endcsname%
      \put(1454,440){\makebox(0,0){\strut{} 1}}%
      \csname LTb\endcsname%
      \put(2127,440){\makebox(0,0){\strut{} 2}}%
      \csname LTb\endcsname%
      \put(2801,440){\makebox(0,0){\strut{} 3}}%
      \csname LTb\endcsname%
      \put(3474,440){\makebox(0,0){\strut{} 4}}%
      \csname LTb\endcsname%
      \put(4148,440){\makebox(0,0){\strut{} 5}}%
      \csname LTb\endcsname%
      \put(4821,440){\makebox(0,0){\strut{} 6}}%
      \put(200,1786){\rotatebox{-270}{\makebox(0,0){\strut{}Hit Probability}}}%
      \put(2800,140){\makebox(0,0){\strut{}Cache Size}}%
    }%
    \gplgaddtomacro\gplfronttext{%
      \csname LTb\endcsname%
      \put(4158,1378){\makebox(0,0)[r]{\strut{}Simulation for $\gamma = 12$}}%
      \csname LTb\endcsname%
      \put(4158,1128){\makebox(0,0)[r]{\strut{}0th order approximation}}%
      \csname LTb\endcsname%
      \put(4158,878){\makebox(0,0)[r]{\strut{}1st order approximation}}%
    }%
    \gplbacktext
    \put(0,0){\includegraphics{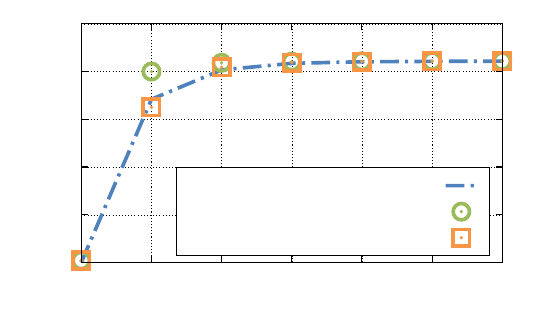}}%
    \gplfronttext
  \end{picture}%
\endgroup
    \caption{\textit{Approximations for $\CatArrRate = 10$}}
    \label{fig:smaller_rate}
  \end{subfigure}
  \caption{\textit{Comparison between hit probability curves obtained in the experiments and their
      analytic approximations in the case of low $\CatArrRate$.}}
\end{figure}

The above numerical results therefore illustrate the accuracy of the
asymptotic expansion for the hit probability.

\section{Concluding Remarks}
\label{sec:conclusion}

In this paper, we have estimated the hit probability of a LRU cache for a
traffic model based on a Poisson cluster point process. In this endeavor, we
have built using Palm theory a probability space where a tagged document can be
analyzed independently from the rest of the process. In the case of the LRU
replacement policy, this property is key for the analysis, since it allowed us
to derive an integral expression for the expected number of misses of the
tagged object.

Using this expression, we were able to obtain an asymptotic expansion of this
integral for large $\CSz$ under the scaling $\CSz = \CatArrRate \AveSojourn$
for fixed $ \AveSojourn>0$. This expansion quantifies rigorously and in
precise fashion the error made when applying the commonly used ``Che
approximation''. We have further shown that the latter expansion is valid for a
sub-class of processes suitable for modeling purposes. Finally, the accuracy
of our theoretical results has been illustrated by numerical experiments.

Our framework could be used to analyze other caching policies satisfying that
the eviction policy for the canonical document depends only on the rest of the
document request process. Examples of such caching policies found in the
literature are RANDOM, which evicts a uniformly chosen document when adding a
new document to the cache, and FIFO, which works as LRU except that it does not
move a requested document that is already in the cache to the front of it. Such
alternative policies may be relevant in that the replacement operations are
somewhat simpler than LRU, and this may compensate their probable lesser
performance in terms of the hit probability. However, the miss events for this
policies are more intricate to analyze, since they depend on the missed requests in
the rest $\MarkCatPP \setminus \delta_{0,\DocReqPP_0}$.


Another possible extension of our study would be to take into account the fact that
documents have random sizes. These sizes and the cache size $\CSz$ should be
measured for instance in bits, packets, or by a continuous value in $\RR^+$.
The document sizes can be incorporated as additional marks to the cluster point
process. In this case, the process $\NDifDocs$ defining the canonical exit time
becomes a compound inhomogeneous Poisson process, summing up these file sizes.
The exit time to consider for a canonical document of size $S$ is then the
first passage time of $X$ strictly above $\CSz-S$.


\section{Proof Section}
\label{sec:proofs}
\subsection{Proof of Proposition~\ref{pro:palm_distribution}}
\label{sec:proof_palm_distribution}

The Slivnyak-Mecke Theorem characterizes the Laplace functional of Poisson point processes under their
Palm distributions, see \cite[Prop.~13.1.VII]{daley2008introduction}.
  Here, for $(u, \nu)$ in $\RR \times \PPSpace{\RR}$, the
  Laplace functional $\LF[][u, \nu]$ of $\tilde{\TotReqPP}$ under the Palm
  distribution $\Palm{u, \nu}$ can be expressed by
  \[
    \LF[f][u, \nu] = e^{-f(u, \nu)} \cdot \LF[f]
  \]
  for any measurable function $f:\RR \times \PPSpace{\RR} \to \RR^+$, where
  $\LF$ is the Laplace functional on the original probability space. The
  Laplace functional $\LF[][u]$ under $\APalm{u}$ is consequently given by
  \[
    \LF[f][u] = \E{\LF[f][u, \DocReqPP_u]} = \E{e^{-f(u, \DocReqPP_u)}} \LF[f].
  \]
  Note that the expectation in the right-hand side is the Laplace functional of
  the point process $\delta_{u, \DocReqPP_u}$. Since Laplace functionals
  characterize point processes, the conclusion follows.

\subsection{Proof of Proposition~\ref{pro:characterization_ndifdocs}}
\label{sec:proof_characterization_ndifdocs}

 Condition~\eqref{eq:finiteness_condition} implies that $\NDifDocsMean^s(t) < \infty$ for all $t \geq s$. 
 Among all points $(\CatArr,\DocReqPP_\CatArr)$ in the rest $\MarkCatPP \setminus \delta_{0,\DocReqPP_0}$,
 the process $(\NDifDocs_u^s)_{s\le u \le t}$ counts those such that $F^s(\DocReqPP_\CatArr)$ falls in $[s,t]$,
 and for $h \ge 0$ the increment $\NDifDocs_{t+h}^s - \NDifDocs_t^s$ counts those
 such that $F^s(\DocReqPP_\CatArr)$ falls in $(t,t+h]$.
 Since the corresponding two subsets of $\RR \times \PPSpace{\RR}$ are disjoint and
 $\MarkCatPP \setminus \delta_{0, \DocReqPP_0}$ is Poisson, we conclude that
 $\NDifDocs^s$ is a counting process with independent increments. In consequence, 
 it is a inhomogeneous Poisson process.

  The mean function for this process is then given by
  \[
    \E{\NDifDocs_t^s}
    =
    \E{\sum_{\CatArr \in \CatPP} \II{F^s(\DocReqPP_\CatArr) \in [s,t]}}
    =
    \E{\sum_{\CatArr \in \CatPP} \II{\DocReqPP_\CatArr[s,t] \geq 1}}.
  \]
  Formula \eqref{eq:ndifdocsmean} follows from the latter expression and the
  fact that the mean measure $\eta$ of $\MarkCatPP \setminus \delta_{0,
    \DocReqPP_0}$ is defined by
  \[
    \eta([t_1, t_2] \times B)
    =
    \CatArrRate \int_{t_1}^{t_2} \PP{\DocReqPP_\CatArr \in B} \dd a
  \]
  for any Borel subset $B$ of $\PPSpace{\RR}$.

\subsection{Proof of Theorem~\ref{pro:expected_number_of_misses}}
\label{sec:proof_expected_number_of_misses}

In the first r.h.s. term of~\eqref{eq:nmiss_cech}, 
$\NReqs$ is a mixed Poisson random variable with random mean $\AveNReqs$ and thus
\begin{equation}
\label{eq:expect-first-term}
\E{\II{\NReqs \geq 1}} = \E{\E{\II{\NReqs \geq 1}}[\AveNReqs]} = \E{1 - e^{-\AveNReqs}} \defEqual \MinAveNMiss\,.
\end{equation}
 Consider now the second r.h.s. term of~\eqref{eq:nmiss_cech}. 
 Following~\cite{kallenberg2006foundations}[p.106 seq.], since the family
  $\ExitTime{\CSz}^s$ for $s \geq 0$ is defined on the rest of the process and thus is independent
  from the request process $\DocReqPP = \sum_{\iReq=1}^\NReqs \delta_{\Req_\iReq}$ for the tagged
  document, 
  \begin{equation}
    \label{eq:representation_ec}
    \E{\II{N>2} \sum_{\iReq=2}^N \II{\Req_\iReq > \ExitTime{\CSz}^{\Req_{\iReq -1}}} \,\bigg|\, \DocReqPP}
    =
    h(\DocReqPP)
  \end{equation}
  where $h:\PPSpace{\RR} \to \RR^+$ is the measurable function defined by
  \begin{equation*}
    h \left( \sum_{\iReq=1}^n \delta_{t_i} \right)
    =
    \II{n>2} \, \E{\sum_{\iReq=2}^n \II{t_\iReq > \ExitTime{\CSz}^{t_{\iReq-1}}}}.
  \end{equation*}
  Since $\ExitTime{\CSz}^s - s \distEqual \ExitTime{\CSz}$, see Proposition~\ref{pro:characterization_ndifdocs}, 
  the function $h$ can be rewritten as
  \[
    h \left( \sum_{\iReq=1}^n \delta_{t_i} \right)
    =
    \II{n>2} \, \E{\sum_{\iReq=2}^n \II{t_\iReq - t_{\iReq-1} > \ExitTime{\CSz}}}.
  \]
  We use this to compute the expectation of the l.h.s. of eq.~\eqref{eq:representation_ec},
  which combined with~\eqref{eq:nmiss_cech} and \eqref{eq:expect-first-term} yields that
  \[
    \E{\NMiss_\CSz}
    =
    \MinAveNMiss + \E{\II{\NReqs>2} \sum_{\iReq=2}^N \II{\Req_\iReq - \Req_{\iReq-1} > \ExitTime{\CSz}}}.
  \]
  Now, since the canonical intensity $\DocReqInt$ and exit time $\ExitTime{\CSz}$
  are independent from the request process of the tagged document,
  Proposition~\ref{pro:holding_times_functionals} yields that
  \begin{align*}
    \E{\NMiss_\CSz}
    &=
    \MinAveNMiss + \E{\int_0^\infty \!\! \dd w \, \II{w > \ExitTime{\CSz}}
      \int_0^\infty \!\! \dd u \, \DocReqInt(u) \DocReqInt(u+w) e^{-(\DocReqMeanFn(u+w)-\DocReqMeanFn(u))}}
    \\
    &=
    \MinAveNMiss
    +
    \E{\int_0^\infty \!\! \DocReqInt(u)
        e^{-(\DocReqMeanFn(u+\ExitTime{\CSz}) - \DocReqMeanFn(u))}\, \dd u
      -
      \int_0^\infty
      \DocReqInt(u) e^{-(\AveNReqs - \DocReqMeanFn(u))} 
      \, \dd u }
    \\
    &=
    \E{\int_0^\infty \!\! \DocReqInt(u)
        e^{-(\DocReqMeanFn(u+\ExitTime{\CSz}) - \DocReqMeanFn(u))}
        \, \dd u
    },
  \end{align*}
  where we use for the last equality that, since $\AveNReqs(\infty) = \AveNReqs$ and $\AveNReqs(0) = 0$,
  \[
    \int_0^\infty \DocReqInt(u) e^{-(\DocReqMeanFn - \DocReqMeanFn(u))} \, \dd u
    =
    \left[ e^{-(\AveNReqs - \AveNReqs(u))} \right]_0^{\infty}
    =
    1 - e^{-\AveNReqs}\,.
  \]
  This last equation and dominated convergence imply that
  $\lim_{t\to\infty} \downarrow m(t) = \E{1 - e^{-\AveNReqs}}$,
  which concludes the proof.

\subsection{Proof of Proposition~\ref{pro:holding_times_functionals}}
\label{sec:proof_holding_times_functionals}

Recall that, given that the process $\DocReqPP$ has $k$ points, the request
times $(\Req_{\iReq})_{\iReq=1}^k$ have the distribution of the order
statistics of a random variable with density $\dens(t) =
\DocReqInt(t)/\AveNReqs$ for $t \geq 0$, and thus with c.d.f.
$\dist$ given by $\dist(t) = \DocReqMeanFn(t)/\AveNReqs$ for $t \geq 0$.
Let $\rdist = 1 - \dist$ denote the complement of $\dist$. From
order statistics theory, it is known that the
holding times $\Req_\iReq - \Req_{\iReq-1}$ for $2\leq \iReq \leq k$ have density
$\tilde{\dens}_{k,r}$ given for $w\geq0$ by
\[
  \tilde{\dens}_{k,r}(w) \defEqual \frac{k!}{(\iReq-2)! (k - \iReq)!}
  \int_0^\infty \dist^{\iReq-2}(u) \dens(u) \dens(u+w) \rdist^{k - \iReq}(u+w) \,\dd u\,.
\]
Consequently, for $k\ge2$ we have
\[
  \E{F(\Req_\iReq - \Req_{\iReq-1})
  }[\NReqs = k]
  = \int_0^\infty F(w) \tilde{\dens}_{k,\iReq}(w) \, \dd w
\]
and hence
\begin{align}
  \label{eq:big_sum}
  \nonumber
  \E{
    \II{\NReqs \geq 2}
    \sum_{\iReq=2}^\NReqs F(\Req_\iReq - \Req_{\iReq -1})
  }
  &=
  \sum_{k=2}^\infty
  \sum_{\iReq=2}^k
  \E{F(\Req_\iReq - \Req_{\iReq -1})}[\NReqs = k]
  e^{-\AveNReqs} \frac{\AveNReqs^k}{k!}
  \\
  &=
  \int_0^\infty F(w) \,
  e^{-\AveNReqs}
  \sum_{k=2}^\infty
  \sum_{\iReq=2}^k
  \tilde{\dens}_{k,\iReq}(w)
  \, \frac{\AveNReqs^k}{k!} \, \dd w\,.
\end{align}
Now, using the Binomial Theorem,
\[
  \sum_{r=2}^k \frac{k!}{(\iReq-2)! (k - \iReq)!}\dist^{\iReq-2}(u) \rdist^{k-\iReq}(u+w)
  =
  k(k-1) [\dist(u) + \rdist(u+w)]^{k-2}
\]
and thus
\[
  \sum_{\iReq = 2}^k
  \tilde{\dens}_{k,w}(w) \frac{\AveNReqs^k}{k!}
  =
  \int_0^\infty
  [\dist(u) + \rdist(u+w)]^{k-2} \frac{\AveNReqs^k}{(k-2)!}
  g(u) g(u+w) \,\dd u\,,
\]
and we conclude that
\[
  e^{-\AveNReqs}\sum_{k=2}^\infty \sum_{\iReq = 2}^k
  \tilde{\dens}_{k,w}(w) \frac{\AveNReqs^k}{k!}
  =
  \AveNReqs^2 \int_0^\infty e^{-\AveNReqs(1 - \dist(u) - \rdist(u+w))} g(u) g(u+w) \,\dd u\,.
\]
Since $
\AveNReqs \times (1 - \dist(u) - \rdist(u+w))
=
\AveNReqs \times (\dist(u+w) - \dist(u))
=
\DocReqMeanFn(u+w) - \DocReqMeanFn(u)
$ and $
g(u)g(u+w) = \DocReqInt(u) \DocReqInt(u+w)/\AveNReqs^2
$,
Equation~\eqref{eq:big_sum} together with the latter intermediate results
concludes the proof.

\subsection{Proof of Proposition~\ref{pro:key_identity}}
\label{sec:proof_key_identity}

  Since the processes $\DocReqInt_\CatArr(\cdot)$ and $\DocReqInt_0(\cdot - \CatArr)$ have the
  same distribution, and $\DocReqMeanFn_\CatArr(0)=0$ for $a\ge0$, we may write $\NDifDocsMean(t)$ as
  \[
    \NDifDocsMean(t)
    =
    \CatArrRate
    \int_{-\infty}^0
    \E{
      1 - e^{-(\DocReqMeanFn(t-\CatArr) - \DocReqMeanFn(-\CatArr))}
    } \dd \CatArr
    +
    \CatArrRate
    \int_0^t
    \E{
      1 - e^{-\DocReqMeanFn(t-\CatArr)}
    } \dd \CatArr\,.
  \]
  We denote the first integral by $I_1(t)$ and the second by $I_2(t)$.
  The change of variables $\CatArr \mapsto -\CatArr$ yields 
  \[
    I_1(t)
    = \int_{-\infty}^0
    \E{
      1 - e^{-(\DocReqMeanFn(t-\CatArr) - \DocReqMeanFn(-\CatArr))}
    } \dd \CatArr
    =
    \int_0^\infty
    \E{
      1 - e^{-(\DocReqMeanFn(t+\CatArr) - \DocReqMeanFn(\CatArr))}
    } \dd \CatArr
  \]
  and thus, using $\frac{\dd}{\dd \CatArr} e^{-(\DocReqMeanFn(t+\CatArr) - \DocReqMeanFn(\CatArr))}
  = - (\DocReqInt(t + \CatArr)-\DocReqInt(\CatArr))e^{-(\DocReqMeanFn(t+\CatArr) - \DocReqMeanFn(\CatArr))}$,
  \begin{align*}
    I_1'(t)
    &=
    \int_0^\infty
    \E{
      \DocReqInt(t + \CatArr)
      e^{-(\DocReqMeanFn(t+\CatArr) - \DocReqMeanFn(\CatArr))}
    }
    \dd \CatArr
    \\
    &=
    \E{
      e^{-\DocReqMeanFn(t)} - 1
    }
    +
    \E{
      \int_0^\infty
      \!\!\!\!\! \DocReqInt(\CatArr)
      e^{-(\DocReqMeanFn(t+\CatArr) - \DocReqMeanFn(\CatArr))}
      \,\dd \CatArr
    }
    \\
    &=\E{
      e^{-\DocReqMeanFn(t)} - 1
    }
    +
    \AveNMiss(t)\,.
  \end{align*}
  Now, the change of variables $\CatArr \mapsto t-\CatArr$ yields
  \[
    I_2(t)
    =
    \int_0^t
    \E{
      1 - e^{-\DocReqMeanFn(t-\CatArr)}
    } \dd \CatArr
    =
    \int_0^t
    \E{
      1 - e^{-\DocReqMeanFn(\CatArr)}
    } \dd \CatArr\,,
  \]
  and hence $I_2'(t) = \E{1-e^{-\DocReqMeanFn(t)}}$. Thus $\NDifDocsMean'(t) = \CatArrRate (I_1'(t) + I_2'(t)) =
  \CatArrRate \,\AveNMiss(t)$ as claimed. We conclude by integrating this, since $\NDifDocsMean(t) =0$.

\subsection{Proof of Lemma~\ref{lem:gamma}}
\label{sec:proof_gamma}

\subsubsection*{\ref{lem:gamma-i}}
 This is the  optimized exponential Markov inequality which is used for
    the upper bound in Cramer's large deviations
    Theorem, see~\cite[Theorem~2.2.3, Remark (c)]{dembo2009large}.
    
\subsubsection*{\ref{lem:gamma-ii}}
Expanding the $k$-th order central moment of $X_\CSz$ in terms
    of the known moments of $\FPTPoiss{\CSz}$ yields that
    \begin{align*}
      \E{(X_\CSz - 1)^k}
      &=
      \sum_{i = 0}^k \binom{k}{i} \frac{\E{(\FPTPoiss{\CSz})^i}}{\CSz^i} (-1)^{k-i}
      \\
      &=
      \rec{\CSz^k}
      \sum_{i = 0}^k
      \binom{k}{i} (-\CSz)^{k-i}
      \frac{\Gamma(\CSz + i)}{\Gamma(\CSz)}
      \\
      &=
      \rec{\CSz^k}\phi_k(\CSz)\,,
    \end{align*}
    where $\phi_k$ is a polynomial of degree at most $k$. As shown
    in~\cite{amm2011march}, the polynomial $\phi_k$ is actually of degree $\lfloor
    k/2 \rfloor$, which allows us to conclude.

\subsection{Proof of Theorem~\ref{thm:asymptotic_expansion}}
\label{sec:asymptotic_expansion}

  Define the function $f_\AveSojourn$ by
  \[
    f_\AveSojourn(z)
    \defEqual
    \AveNMiss(\IAveNMiss^{-1}(\AveSojourn z))
    =
    \AveNMiss(\ChT{\AveSojourn z}).
  \]
  With the scaling $\CSz = \CatArrRate \AveSojourn$,
  Equation~\eqref{eq:ave_nmiss_change_var2} can be then written as
  \begin{equation}
    \label{eq:ave_nmiss_scaling}
    \E{\NMiss_\CSz}
    =
    \E{f_\AveSojourn \left( \frac{\FPTPoiss{\CSz}}{\CSz} \right)}.
  \end{equation}
  Let again $X_\CSz \defEqual \FPTPoiss{\CSz}/\CSz$ as in Lemma~\ref{lem:gamma},
  and fix $\eta > 0$. Let us decompose the expectation~\eqref{eq:ave_nmiss_scaling} into
  $\E{\NMiss_\CSz} = A_\CSz + B_\CSz$ where
  \[
    A_\CSz \defEqual \E{f_\AveSojourn(X_\CSz) \I_{\{| X_\CSz - 1 | \geq \eta\}}}
    ,\quad
    B_\CSz \defEqual \E{f_\AveSojourn(X_\CSz) \I_{\{| X_\CSz - 1 | < \eta\}}}.
  \]
  
  For $A_\CSz$, recall
  that the function $\AveNMiss$ is bounded by $\E{\AveNReqs} < \infty$, and so is
  $f_\AveSojourn$. Then, by Lemma~\ref{lem:gamma}~\ref{lem:gamma-i}, we have
  \[
    A_\CSz
    \leq
    \E{\AveNReqs} \PP{\left| X_\CSz - 1 \right| \geq \eta}
    \leq
    2\E{\AveNReqs}
    e^{-\CSz \cdot \varphi(1+\eta)}
    = o(1/\CSz)\,.
  \]
  
  For $B_\CSz$, we write a Taylor expansion of $f_\AveSojourn$ at 1 of order two in the form
  \begin{align*}
    f_\AveSojourn(X_\CSz)
    &=
    f_\AveSojourn(1)
    +
    f_\AveSojourn'(1) \left( X_\CSz-1 \right)
    +
    \frac{f_\AveSojourn''(Y_\CSz)}{2} \left(X_\CSz-1 \right)^2
    \\
    &=
    h_\AveSojourn(X_\CSz)
    +
    k_\AveSojourn(X_\CSz, Y_\CSz)\,,
  \end{align*}
  where $Y_\CSz$ is a random variable in the random interval $[1, X_\CSz]\cup [X_\CSz,1]$, and
  \[
    \begin{cases}
      \displaystyle
      h_\AveSojourn(X_\CSz)
      \defEqual
      f_\AveSojourn(1)
      +
      f_\AveSojourn'(1) \left( X_\CSz-1 \right)
      +
      \frac{f_\AveSojourn''(1)}{2} \left(X_\CSz-1 \right)^2,
      \\
      \displaystyle
      k_\AveSojourn(X_\CSz, Y_\CSz) = \frac{f_\AveSojourn''(Y_\CSz) -f_\AveSojourn''(1)}{2} \left(X_\CSz - 1 \right)^2.
    \end{cases}
  \]
  Then $B_\CSz = D_\CSz + E_\CSz$
  where
  \[
    D_\CSz = \E{h_\AveSojourn(X_\CSz) \I_{\{| X_\CSz - 1 | < \eta\}}}
    ,\quad
    E_\CSz = \E{k_\AveSojourn(X_\CSz, Y_\CSz) \I_{\{| X_\CSz - 1 | < \eta\}}}.
  \]
  We then compute
  \begin{equation}
    \label{eq:h_gdv}
    D_\CSz
    =
    \E{h_\AveSojourn(X_\CSz)} - \E{h_\AveSojourn(X_\CSz) \I_{\{| X_\CSz - 1 | \geq \eta\}}}
  \end{equation}
  where
  \[
    \E{h_\AveSojourn(X_\CSz)} = f_\AveSojourn(1) + \frac{f_\AveSojourn''(1)}{2\CSz}
  \]
  since $\E{X_\CSz - 1} = 0$ and $\E{(X_\CSz - 1)^2} = 1/\CSz$. Besides, to
  deal with the second term $\E{h_\AveSojourn(X_\CSz) \I_{\{| X_\CSz - 1 | \geq \eta\}}}$
  in the right-hand side of~\eqref{eq:h_gdv}, we use the Cauchy-Schwarz inequality
  to write
  \[
    \left| \E{h_\AveSojourn(X_\CSz) \I_{\{| X_\CSz - 1 | \geq \eta\}}} \right|
    \leq
    \sqrt{\E{h_\AveSojourn(X_\CSz)^2}} \sqrt{\PP{| X_\CSz - 1 | \geq \eta}}
  \]
  and note that $\E{h_\AveSojourn(X_\CSz)^2} = O(1)$ for all $\CSz > 1$ by
  Lemma~\ref{lem:gamma}~\ref{lem:gamma-ii}. Applying Lemma~\ref{lem:gamma}~\ref{lem:gamma-i} then
  eventually shows that $\E{h_\AveSojourn(X_\CSz) \I_{\{| X_\CSz - 1 | \geq \eta\}}}$ is
  $O(e^{-\frac{\CSz}{2} \cdot \varphi(1+\eta)})$ which is, in particular,
  $o(1/\CSz)$. At this stage, we therefore conclude from~\eqref{eq:h_gdv} and
  the latter discussion that
  \begin{equation}
    \label{eq:mid_stage}
    D_\CSz
    =
    f_\AveSojourn(1) + \frac{f_\AveSojourn''(1)}{2\CSz} + o\left(\frac{1}{\CSz} \right).
  \end{equation}

  Lastly, we show that the term $E_\CSz$ is $o(1/\CSz)$. To this aim, it is
  sufficient to show that
  the sequence $W_\CSz = \CSz \cdot k_\AveSojourn(X_\CSz, Y_\CSz)$ for $\CSz > 1$
  converges in probability to zero and that it is uniformly integrable (\cite[Theorem 13.7]{williams1991probability}).

  \bull To prove the convergence in probability, note that since $X_\CSz \to 1$
  a.s. when $\CSz \to \infty$ and $Y_\CSz \in [1, X_\CSz]\cup [X_\CSz,1]$, then $Y_\CSz \to
  1$ a.s. It follows from the continuity of $f_\AveSojourn''$ in the interval
  $(1-\eta, 1+\eta)$ that $f_\AveSojourn''(1) - f_\AveSojourn''(Y_\CSz) \to 0$ a.s. and, in
  particular, in probability. On the other hand, since $X_\CSz =
  \FPTPoiss{\CSz}/\CSz$ is an average of $\CSz$ \iid random variables with mean
  1, the continuous mapping theorem for weak limits implies that $C(X_\CSz -
  1)^2$ converges in distribution (the limit distribution is $\chi^2$ with
  parameter $1$ but this specific limit has no importance for the present
  proof). Finally, since $\II{\left| X_\CSz - 1 \right| < \eta} \to 1$ a.s.,
  Slutsky's theorem (\cite[Th. 11.4]{gut2006probability}) allows us to
  conclude that
  \[
    W_\CSz = \frac{f_\AveSojourn''(1) - f_\AveSojourn''(Y_\CSz)}{2} \times
    \CSz(X_\CSz - 1)^2 \times \II{\left| X_\CSz - 1 \right| < \eta} \to 0
  \]
  in distribution as $\CSz \to \infty$, and thus in
  probability as well.

  \bull To prove the uniform integrability of $W_\CSz$, it suffices to show that
  \begin{equation}
    \label{eq:ineq_ui}
    \sup_{\CSz \geq 1} \E{W_\CSz^2} < \infty
  \end{equation}
  (see~\cite[Theorem 13.3]{williams1991probability}).
  Since $f_\AveSojourn$ is twice continuously differentiable,
  \[
    \left|
      \frac{f_\AveSojourn''(1) -f_\AveSojourn''(Y_\CSz)}{2}
      \I_{\{| X_\CSz - 1 | < \eta \}}
    \right|
    \leq
    K
  \]
  for any $\CSz > 1$ and for some constant $K$ depending on $\eta$ only.
  By Lemma~\ref{lem:gamma}, we further have $\E{\CSz^2 \left(X_\CSz -
      1 \right)^4} = \CSz^2 \times O(\CSz^{-2}) = O(1)$. We finally conclude
  that $\E{W_\CSz^2} < K^2 \times O(1) < \infty$, which proves the claimed
  property~\eqref{eq:ineq_ui}.

  Finally gathering $\E{\NMiss_\CSz} = A_\CSz + B_\CSz = A_\CSz + D_\CSz +
  E_\CSz$ with $A_\CSz = o(1/\CSz)$, $E_\CSz = o(1/\CSz)$ and $D_\CSz$ expanded
  in~\eqref{eq:mid_stage}, we thus have proved that
  \begin{equation}
    \label{eq:asymptotic_expansion_f}
    \E{\NMiss_\CSz}
    =
    f_\AveSojourn(1) + \frac{f_\AveSojourn''(1)}{2\CSz} + o\left(\rec{\CSz} \right)
  \end{equation}
  as $\CSz \to \infty$.
  To conclude the proof, we now express the function $f_\AveSojourn$ and its derivatives
  at 1 in terms of function $\AveNMiss$ and its derivatives at $\ChT{\AveSojourn}$. By
  implicit differentiation,
  \[
    f_\AveSojourn'(z)
    =
    \frac{\AveNMiss'(\ChT{\AveSojourn z})}{\AveNMiss(\ChT{\AveSojourn z})} \AveSojourn
    ,\quad
    f_\AveSojourn''(z) = \frac{\AveSojourn^2}{\AveNMiss(\ChT{\AveSojourn z})^2}
    \left(
      \AveNMiss''(\ChT{\AveSojourn z})
      -
      \frac
      {\AveNMiss'(\ChT{\AveSojourn z})^2}
      {\AveNMiss(\ChT{\AveSojourn z})}
    \right),
  \]
  and the values of $f_\AveSojourn'$ and $f_\AveSojourn''$ at $z = 1$ consequently follow.
  Replacing them into~\eqref{eq:asymptotic_expansion_f}, we finally prove the
  expansion~\eqref{eq:asymptotic_expansion}, as claimed.

\subsection{Proof of Proposition~\ref{pro:scale_family}}
\label{sec:differentiability}
Differentiating \eqref{eq:ave_nmiss_ttl} under the integral sign, with
$\DocReqInt(u)$ expressed by \eqref{eq:scale_family_def}, readily gives
formulas~\eqref{eq:derivatives} after using the change of variables $ u \mapsto
u/\Lifespan$. The validity of these formulas can then be simply proved by showing that
these integrals for $\AveNMiss'$ and $\AveNMiss''$ are finite.

Given $t > 0$ and $\Lifespan$, define $u^* = u^*(t, \Lifespan) = \inf \lbrace u:
f(u) > f(u + t/\Lifespan) \rbrace$, so that $f(u) \leq f(u + t/\Lifespan)$ for
$u \leq u^*$ and $f(u) > f(u + t/\Lifespan)$ for $u > u^*$. The existence of
$u^*$ is ensured from the unimodality of $f$, and we have $u^* = 0$ if and only if
$f$ is non-increasing. Finally, define $\tilde{u} = \inf \lbrace u : f(u) =
\max f\rbrace$ (see Fig.~\ref{fig:unimodal_schema} for a schematic view
of these definitions).
\begin{figure}[!htb]
  \centering
  \includegraphics[scale=0.4]{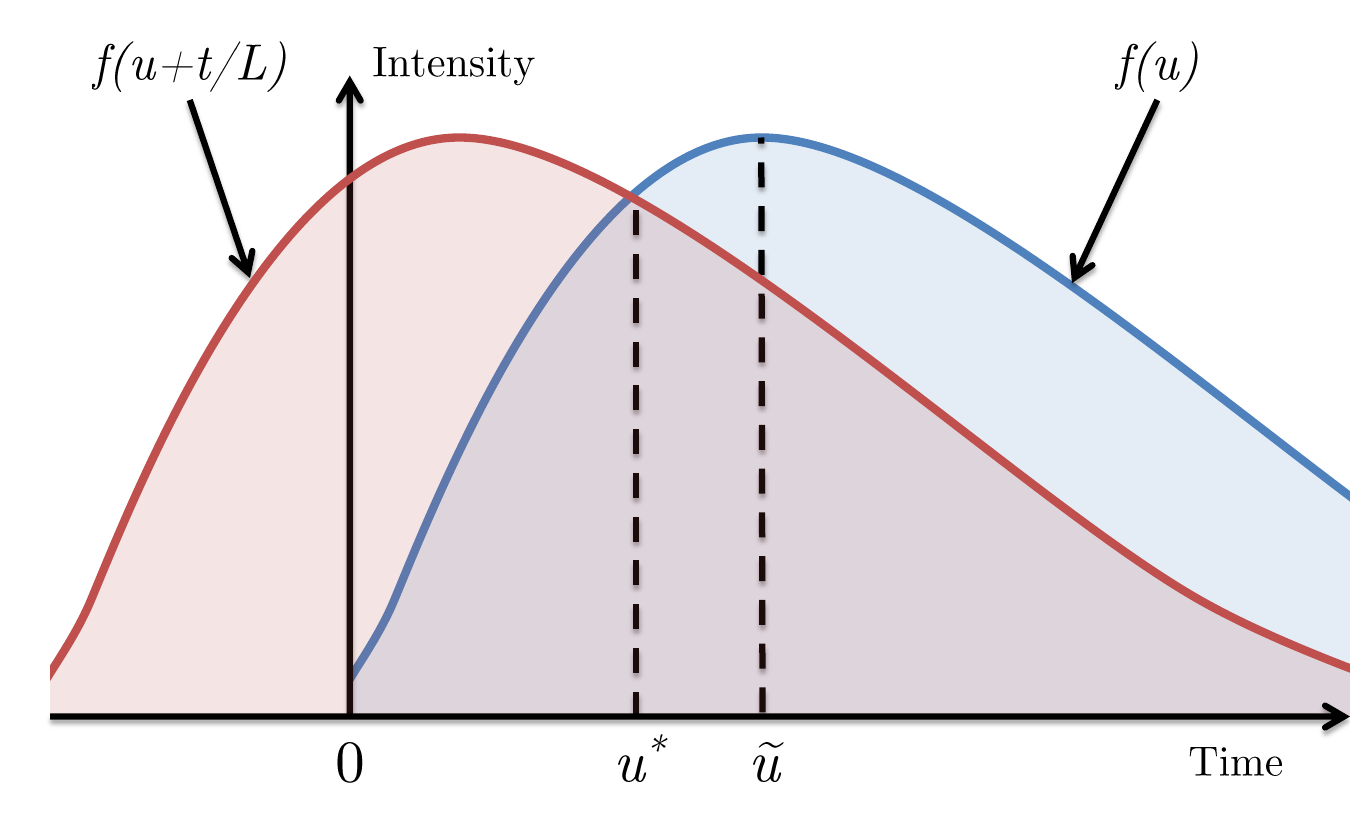}
  \caption{\textit{Schema for unimodal $f$}}
  \label{fig:unimodal_schema}
\end{figure}

Since $f$ is differentiable and unimodal, it is quasi-concave
(see~\cite[Lemma 2.4.1]{dos2013quasiconvex}), that is, for any $0 \leq \eta
\leq 1$ we have $f(\eta u_1 + (1 - \eta) u_2) \geq f(u_1) \mini f(u_2)$ for
$u_1,u_2 \geq 0$. As a consequence, for any $t > 0$, the area under the graph
of $f$ in the interval $[u, u+t/\Lifespan]$ can be bounded below by
\begin{equation}
  \label{eq:lower_bounds_F}
  F(u + t/\Lifespan) - F(u)
  \geq
  \begin{cases}
    f(u) \cdot t / \Lifespan, & \quad u \leq u^*, \\
    f(u + t/\Lifespan) \cdot t/\Lifespan, & \quad u > u^*.
  \end{cases}
\end{equation}

We now partition the integrals in \eqref{eq:derivatives} into their
contributions from intervals $[0, u^*]$ and $[u^*, \infty)$, respectively, and
bound them separately.
For the first derivative $\AveNMiss(t)$, the lower bounds~\eqref{eq:lower_bounds_F} yield
\begin{align*}
  |\AveNMiss'(t)|
  &\leq
  \E{
    \Pop \Lifespan
    \int_0^{u^*}
    f(u+t/\Lifespan) \Pop f(u) e^{-\Pop f(u) t} \, \dd u
  }
  \\
  &+
  \E{
    \Pop \Lifespan
    \int_{u^*}^\infty
    f(u) \Pop f(u+t/\Lifespan) e^{-\Pop f(u + t/\Lifespan) t} \, \dd u
  } \leq \frac{2}{et} \E{\Pop \Lifespan}
\end{align*}
where the last inequality is justified by the bound $xe^{-ax} \leq 1/ae$ for any fixed
$a > 0$, and
the fact that $\int f = 1$.

For the second derivative $\AveNMiss''(t)$, we introduce the integrals
\begin{align*}
  A_1(t)
  &=
  \E{\Pop \Lifespan \int_0^\infty \Pop^2 f(u)f(u+t/\Lifespan)^2 e^{-\Pop \Lifespan ( F(u + t/\Lifespan) - F(u)) \, \dd u}},
  \\
  A_2(t)
  &=
  \E{\Pop \int_0^\infty \Pop f(u)f'(u+t/\Lifespan) e^{-\Pop \Lifespan ( F(u + t/\Lifespan) - F(u)) \, \dd u}}
\end{align*}
so that $|m''(t)| \leq |A_1(t)| + |A_2(t)|$. For $A_1(t)$, we have
\begin{align*}
  |A_1(t)|
  &\leq
  \E{
    \Pop \Lifespan \int_0^{u^*}
    f(u+t/\Lifespan)^2 f(u) \Pop^2 e^{-\Pop f(u) t} \, \dd u
  }
  \\
  &+
  \E{
    \Pop \Lifespan \int_{u^*}^\infty
    f(u) \Pop^2 f(u+t/\Lifespan)^2 e^{-\Pop f(u + t/\Lifespan) t} \, \dd u
  }
  \\
  &\leq
  \E{
    \frac{ 4 \Pop \Lifespan }{e^2 t^2 f(0)} \int f^2
  }
  +
  \E{
    \frac{\Pop \Lifespan}{e t}
  }
  \leq \rec{et} \left( 1 + \frac{4}{f(0)et} \int f^2 \right) \E{\Pop \Lifespan}
  < \infty
\end{align*}
where the last inequality follows from the bounds $x e^{-ax} \leq 1/ae$,
$x^2 e^{-ax} \leq 4/a^2e^2$ for any fixed $a > 0$, and the fact that $0 < f(0) \leq f(u) \leq f(u+t/\Lifespan)$ for $u \in [0,
u^*]$. Regarding $A_2(t)$, we have
\begin{align*}
  |A_2(t)|
  &\leq
  \E{\Pop \int_0^{u^*} \Pop f(u)|f'(u+t/\Lifespan)| e^{-\Pop f(u) t} \, \dd u}
  \\
  &+
  \E{\Pop \int_{u^*}^\infty \Pop f(u) |f'(u+t/\Lifespan)| e^{-\Pop f(u + t/\Lifespan) t} \, \dd u}
  \\
  &= B_1(t) + B_2(t).
\end{align*}
Using again $xe^{-ax} \leq 1/ae$, we have
\[
  B_1(t) \leq \frac{\E{\Pop}}{et} \int |f'| < \infty.
\]
Finally, to deal with $B_2(t)$ we note that $f'(u+t/\Lifespan) \leq 0$ for $u
\in [u^*, \infty)$ and thus $|f'(u + t/\Lifespan)| = -f'(u + t/\Lifespan)$. We
then use integration by parts to obtain
\begin{align*}
  B_2(t)
  &=
  -
  \rec{t}
  \E{\Pop
    \left(
      \left[-e^{-\Pop f(u + t/\Lifespan) t }f(u) \right]_{u=u^*}^\infty
      +
      \int_{u^*}^\infty f'(u) e^{-\Pop f(u+t/\Lifespan) t} \, \dd u
    \right)
  }
  \\
  &=
  -
  \rec{t}\E{\Pop f(u^*)e^{-\Pop f(u^* +t/\Lifespan)}}
  \\
  &\quad -\rec{t}
  \E{\Pop \int_{u^*}^{\tilde{u}} f'(u) e^{-\Pop f(u+t/\Lifespan) t} \, \dd u}
  -
  \rec{t} \E{\Pop \int_{\tilde{u}}^\infty \!\! f'(u) e^{-\Pop f(u+t/\Lifespan) t} \, \dd u}.
\end{align*}
The first term in the latter expression is trivially negative; the second is
also negative since $f$ is non-decreasing in $[0, \tilde{u})$. As a consequence
both terms can be ignored to obtain
\begin{align*}
  B_2(t) \leq
  \rec{t} \E{\Pop \int_{\tilde{u}}^\infty |f'(u)| e^{-\Pop f(u+t/\Lifespan) t} \, \dd u}
  \leq \frac{\E{\Pop}}{t} \int |f'| < \infty
\end{align*}
thus concluding the proof.

\section*{Acknowledgements}
The authors wish to thank Bruno Kauffmann at Orange Labs for his deep insight
and fruitful discussions along with Byron Schmuland for pointing out
reference~\cite{amm2011march} in
\href{https://math.stackexchange.com/questions/1383069/degree-of-polynomial-in-centered-moments-of-gamman-1}{this
  question} posed at the \emph{Mathematics StackExchange} website.

\bibliographystyle{acmtrans-ims}

\end{document}